\documentclass[11pt,a4paper]{article}
\usepackage{amssymb}
\usepackage{amsmath}
\usepackage{amsfonts}
\usepackage{bbm}
\usepackage{amsthm}
\usepackage{mathrsfs}
\usepackage{hyperref}
\usepackage{marvosym}
\usepackage{color}
\usepackage[margin=2.6cm]{geometry}
\usepackage[all,cmtip]{xy}
\usepackage[utf8]{inputenc}

\usepackage{enumitem}

\usepackage{upgreek}

\definecolor{darkred}{rgb}{0.8,0.1,0.1}
\hypersetup{
     colorlinks=false,         
     linkcolor=darkred,
     citecolor=blue,
}

\theoremstyle{plain}
\newtheorem{theo}{Theorem}[section]
\newtheorem{lem}[theo]{Lemma}
\newtheorem{propo}[theo]{Proposition}
\newtheorem{cor}[theo]{Corollary}

\theoremstyle{definition}
\newtheorem{defi}[theo]{Definition}
\newtheorem{ex}[theo]{Example}

\newtheorem{rem}[theo]{Remark}

\numberwithin{equation}{section}

\def\bbR{\mathbb{R}}

\def\bbZ{\mathbb{Z}}
\def\bbT{\mathbb{T}}
\def\bbS{\mathbb{S}}
\def\bbB{\mathbb{B}}

\def\Hom{\mathrm{Hom}}

\def\id{\mathrm{id}}
\def\supp{\mathrm{supp}}

\def\dd{\mathrm{d}}

\def\1{\mathbbm{1}}

\def\op{\mathrm{op}}

\newcommand{\ips}[2]{\langle #1,#2\rangle}

\def\Man{\mathsf{Man}}
\def\oMan{\mathsf{oMan}}
\def\Pair{\mathsf{Pair}}
\def\PePair{\mathsf{Pair}_\mathsf{pe}}
\def\DSet{\mathsf{DSet}}
\def\Ch{\mathsf{Ch}}
\def\Cat{\mathsf{Cat}}
\def\Ab{\mathsf{Ab}}

\def\OmegaZ{\Omega_\bbZ}
\def\free{\mathrm{free}}
\def\tor{\mathrm{tor}}
\def\dR{\mathrm{dR}}

\def\H{\mathrm{H}}
\def\dH{\hat{\H}}
\def\Ht{\H_\tor}
\def\Hf{\H_\free}

\def\cu{\mathrm{curv}}

\def\ch{\mathrm{char}}

\def\del{\partial}
\def\cdel{\updelta}

\def\im{\mathrm{im\,}}

\def\colim{\mathrm{colim}}

\def\emb{\hookrightarrow}

\def\mcK{\mathcal{K}}
\def\mcO{\mathcal{O}}

\def\mcU{\mathcal{U}}

\def\sk{\vspace{2mm}}

\title{%
Cheeger-Simons differential characters\\ 
with compact support and Pontryagin duality
}

\author{%
Christian Becker$^{1,2,a}$, Marco Benini$^{1,3,b}$, Alexander Schenkel$^{3,4,c}$ and Richard J.\ Szabo$^{3,d}$ \vspace{4mm}\\
{\small $^1$ Institut f\"ur Mathematik, Universit\"at Potsdam,}\\
{\small Karl-Liebknecht-Str.~24-25, 14476 Potsdam, Germany.}\vspace{2mm}\\
{\small $^2$ Institut f\"ur Mathematik und Informatik, Universit\"at Greifswald,}\\
{\small Walther-Rathenau-Str.~47, 17489 Greifswald, Germany.}\vspace{2mm}\\
{\small $^3$ Department of Mathematics, Heriot-Watt University,}\\
{\small Colin Maclaurin Building, Riccarton, Edinburgh EH14 4AS, United Kingdom.}\vspace{0.2mm}\\
{\small  Maxwell Institute for Mathematical Sciences, Edinburgh, United Kingdom.}\vspace{0.2mm}\\
{\small  Higgs Centre for Theoretical Physics, Edinburgh, United Kingdom.}\vspace{2mm}\\
{\small $^4$ School of Mathematical Sciences, University of Nottingham,}\\
{\small University Park, Nottingham NG7 2RD, United Kingdom.}\vspace{4mm}\\
{\footnotesize \texttt{Email:} $^a$\texttt{christian.becker@uni-greifswald.de},
$^b$\texttt{mbenini87@gmail.com}, }\\
{\footnotesize $^c$\texttt{alexander.schenkel@nottingham.ac.uk}, $^d$\texttt{R.J.Szabo@hw.ac.uk}}
 }

\date{December 2019}

\begin{document}

\maketitle

\begin{abstract}
\noindent
By adapting the Cheeger-Simons approach to differential cohomology, we establish a notion of differential cohomology with compact support. We show that it is functorial with respect to open embeddings and that it fits into a natural diagram of exact sequences which compare it to compactly supported singular cohomology and differential forms with compact support, in full analogy to ordinary differential cohomology. We prove an excision theorem for differential cohomology using a suitable relative version. Furthermore, we use our model to give an independent proof of Pontryagin duality for differential cohomology recovering a result of [Harvey, Lawson, Zweck -- Amer.\ J.\ Math.\ {\bf 125} (2003) 791]: On any oriented manifold, ordinary differential cohomology is isomorphic to the smooth Pontryagin dual of compactly supported differential cohomology. For manifolds of finite-type, a similar result is obtained interchanging ordinary with compactly supported differential cohomology.
\end{abstract}

\paragraph*{Report no.:} EMPG--15--16

\newpage

\bigskip
{\baselineskip=12pt
\tableofcontents
}
\bigskip


\section{Introduction and summary}
The purpose of this paper is to develop a notion of differential 
cohomology with compact support which is based on the 
Cheeger-Simons approach to differential cohomology \cite{CS}. 
To obtain this result, we consider a suitable model for relative differential cohomology. 
Our techniques allow us achieve two main goals: 
1)~We prove an excision theorem for our model of relative differential cohomology, 
see Theorem \ref{thmDiffCharExcision}.
2)~Using our model, which is closer to the original approach to differential cohomology \cite{CS}, 
we independently recover the results first established by \cite{HLZ} 
about smooth Pontryagin duality for differential cohomology 
without resorting to the theory of de Rham-Federer currents. 
In this respect, our Theorem \ref{thmCDiffCharIso} is equivalent to \cite[Theorem 8.7]{HLZ}, 
where smooth Pontryagin duality is obtained 
using the de Rham-Federer model for differential cohomology. 
Along the way, we also clarify the relation between compactly supported 
and unrestricted Cheeger-Simons differential characters, 
eventually amending an error in \cite[Theorem 8.4]{HLZ} 
(which however does not affect the other results of that paper), 
see Remarks \ref{remHLZ} and \ref{remNonSub}. 
\sk

Differential cohomology is a family of
functors $\dH^k(-;\bbZ): \Man^\op \to \Ab$, for $k \geq 0$,  
from the category of smooth manifolds to the category of Abelian groups that comes together with four natural 
transformations comparing it with certain groups of differential forms and 
also with smooth singular cohomology (with coefficients in $\bbZ$ and $\bbT=\bbR/\bbZ$).
The differential cohomology groups $\dH^k(M;\bbZ)$ of 
a smooth manifold $M$ may thus be regarded as a refinement of the smooth
singular cohomology groups $\H^k(M;\bbZ)$ with integer coefficients 
by differential forms. As a classical geometric example, the second differential cohomology group $\dH^2(M;\bbZ)$ 
is a refinement of the Picard group of isomorphism classes of  
Hermitean line bundles on $M$ to the group of isomorphism classes of  Hermitean line bundles with connection.
From this perspective, it seems natural to expect differential cohomology 
with compact support to be a family of functors $\dH^k_{\rm c}(-;\bbZ):\Man_{m,\emb} \to \Ab$, for
$k\geq 0$, from the category of $m$-dimensional manifolds with smooth open embeddings as morphisms 
to the category of Abelian groups that comes together with natural transformations 
comparing it to compactly supported differential forms and cohomology with compact support.
The notion of compactly supported differential cohomology that we develop in this paper
is easily seen to satisfy this property: It is functorial with respect to smooth open embeddings and it comes 
together with four natural transformations similar to the ones of ordinary differential cohomology.
Thus the compactly supported differential cohomology groups $\dH_{\rm c}^k(M;\bbZ)$ 
of a manifold refine the compactly supported cohomology groups $\H^k_{\rm c}(M;\bbZ)$ 
with $\bbZ$-coefficients by compactly supported differential forms.
In the classical geometric example, the group $\dH_{\rm c}^2(M;\bbZ)$ describes
isomorphism classes of  Hermitean line bundles with connection on $M$ and 
a parallel section outside some compact subset.
\sk

By now there exist several different models for differential cohomology, 
i.e.~several different constructions of the functors $\dH^k(-;\bbZ): \Man^\op \to \Ab$, 
given by differential characters~\cite{CS}, de Rham-Federer currents \cite{HLZ}, 
differential cocycles \cite{HS} and smooth Deligne cohomology \cite{Brylinski}.
All models of differential cohomology are known to be related by unique natural isomorphisms. 
Several instances of such natural isomorphisms have been constructed in \cite{HL1, HLZ}. 
More recently, \cite{SS} introduced an axiomatic approach to differential cohomology, 
showing uniqueness up to unique natural equivalence. 
A similar result has been obtained in \cite{BB}, however using a different approach. 
In the present paper we adopt the original model of differential characters as 
established by Cheeger and Simons in \cite{CS}.
Cheeger-Simons differential characters are $\bbT$-valued group characters on the Abelian group $Z_{k-1}(M)$ 
of singular cycles in $M$ that satisfy a certain smoothness condition. 
An interesting alternative perspective would be to use the 
more abstract homotopy theoretical approach to differential cohomology, 
see e.g.\ \cite{Bu12}. We expect that some of our results,
e.g.\ on functorial properties of various constructions and the excision theorem for differential cohomology,
are intrinsic properties of this framework and could be shown with less arguments.
However, we decided to use the more traditional Cheeger-Simons approach \cite{CS} 
in order to avoid the rather technical tools of homotopy theory. 
This gives us also the chance to relate the notion of compactly supported differential characters 
that we introduce with the original Cheeger-Simons characters, see Section \ref{secCDiffChar}. 
\sk

To introduce differential characters with compact support, we follow the 
well-known construction of compactly supported cohomology as the colimit 
of the relative cohomology functor $\H^\sharp(M,M\setminus -;G)$, with $G$ an Abelian group, 
over the directed set $\mcK_M$ of compact subsets of $M$.
For this construction we need an appropriate notion of differential cohomology relative to a 
smooth submanifold $S \subseteq M$ (possibly with boundary).
As explained in \cite{BB}, there are two different such notions in the realm of 
differential characters: They arise from the two different ways to define relative 
(de Rham and singular) cohomology as the cohomology of the mapping cone complex 
of the inclusion $i_S:S \hookrightarrow M$ or as the cohomology of the subcomplex of 
forms (or cochains) vanishing outside $S \subseteq M$.
Thus relative differential cohomology may be defined as the group of differential 
characters on either cycles of the mapping cone complex or on relative cycles.
Some confusion might arise from the fact that what are called 
\emph{relative differential characters} in the literature \cite{BrT,BB,FeR,B14} 
are not differential characters on relative cycles but characters on mapping cone cycles.
In \cite{BB} it is shown that the group of differential characters on relative cycles 
$\dH^k(M,S;\bbZ)$ is a subgroup of the group of relative differential characters.
Elements of this subgroup are called \emph{parallel relative differential characters} 
in \cite{BB} for geometric reasons. To avoid confusion, in the present paper we will 
only use the groups $\dH^k(M,S;\bbZ)$ of \emph{differential characters on relative cycles}.
\sk

Part of the present paper generalizes some results from \cite{BB} for differential characters on relative cycles
to a less restrictive setting: We allow arbitrary embedded submanifolds (possibly with boundary), 
whereas in \cite{BB} only closed submanifolds are taken into account. Restricting the consideration to properly 
embedded submanifolds, we immediately recover the results from \cite{BB} with 
basically the same arguments.
In the course of introducing differential cohomology with compact support we also 
establish the excision property for relative differential cohomology:
For an open subset $O \subseteq M$ and a closed subset $C \subseteq M$ such that 
$C \subseteq O$, the morphism $(O,O \setminus C) \to (M,M \setminus C)$ in 
the category of pairs induces an isomorphism 
\begin{equation}
\dH^k(M,M \setminus C;\bbZ) \simeq \dH^k(O,O \setminus C;\bbZ)
\end{equation}
in differential cohomology. To the best of our knowledge, this property has 
not been discussed rigorously in the literature so far, although it might have been 
conjectured by experts in the field.
This isomorphism is important for establishing functoriality $\dH^k_{\rm c}(-;\bbZ):\Man_{m,\emb} \to \Ab$
of (our model for) compactly supported differential cohomology with respect to open embeddings of $m$-dimensional manifolds.
\sk

This paper is organized as follows: In Section \ref{secPrel} 
we review some requisite preliminaries
on smooth singular (co)homology groups and their relative versions, compactly supported cohomology
and the Cheeger-Simons model for differential cohomology.
\sk

In Section \ref{secRelDiffChar} we introduce and study 
differential characters on relative cycles, which provide the model for relative differential cohomology used in this paper.
We prove in Theorem \ref{thmRelDiffChar} that, for any submanifold $S\subseteq M$ (possibly with boundary),
the group of differential characters on relative cycles $\dH^k(M,S;\bbZ)$ fits
into a commutative diagram involving a short exact sequence. For generic submanifolds $S$,
this diagram is an incomplete analog of the usual diagram for 
(absolute) differential cohomology. In Theorem \ref{eqRelDiffCharDiaAlt} we shall prove that
for the case where $S\subseteq M$ is properly embedded, the incomplete diagram
can be extended to a full diagram of short exact sequences.
The incomplete diagram in Theorem \ref{thmRelDiffChar} is however sufficient to
prove an excision theorem for differential cohomology in Subsection \ref{subRelDiffCharNat}.
We then show in Subsection~\ref{subRelDiffCharModule} that the graded group of differential characters on relative cycles
$\dH^\sharp(M,S;\bbZ)$ is a module over the differential cohomology ring $\dH^\sharp(M;\bbZ)$.
\sk

In Section \ref{secCDiffChar} we define the groups $\dH^k_{\rm c}(M;\bbZ)$ of differential characters
with compact support as a colimit of the relative differential cohomology functor.
In Theorem \ref{thmCDiffChar} we obtain an analogue of the usual differential cohomology diagram for the compactly supported case. 
We further prove that compactly supported differential cohomology forms a family of 
functors $\dH^k_{\rm c}(-;\bbZ):\Man_{m,\emb} \to \Ab$, for
$k\geq 0$, from the category of $m$-dimensional manifolds with smooth open embeddings as morphisms 
to the category of Abelian groups and that $\dH^\sharp_{\rm c}(M;\bbZ)$ is a module
over the differential cohomology ring $\dH^\sharp(M;\bbZ)$. 
We compare our construction of compactly supported differential characters
with earlier results in \cite{HLZ}, see in particular Remark \ref{remHLZ}.
\sk

In Section \ref{secPontrDuality} we establish smooth Pontryagin duality for differential cohomology.
Similar results were proven in \cite{HLZ} by using (compactly supported) de Rham-Federer characters.
The main results here are the following Pontryagin dualities which are proven in Theorem \ref{thmCDiffCharIso}: 
For any oriented $m$-dimensional manifold $M$, we obtain a natural isomorphism
\begin{equation}
\dH^{m-k+1}(M;\bbZ) \stackrel{\simeq}{\longrightarrow} \dH^k_{\rm c}(M;\bbZ)^\star_\infty~
\end{equation}
between the differential cohomology group $\dH^{m-k+1}(M;\bbZ)$
and the smooth Pontryagin dual $\dH^k_{\rm c}(M;\bbZ)^\star_\infty$ of the compactly supported
differential cohomology group $\dH^k_{\rm c}(M;\bbZ)$.
If moreover $M$ is of finite-type, we can also interchange 
the roles of ordinary and compactly supported differential cohomology to get isomorphisms
\begin{equation}
\dH^{m-k+1}_{\rm c}(M;\bbZ) \stackrel{\simeq}{\longrightarrow} \dH^k(M;\bbZ)^\star_\infty~.
\end{equation}
In~\cite{BBSS} these results are applied to analyze
dualities in (higher) quantum Abelian gauge theories and to the quantization of self-dual fields.


\section{Cohomological preliminaries}\label{secPrel}
We briefly recall some background material which is used extensively throughout the rest of this paper, including smooth singular (co)homology together with their relative versions, 
a colimit prescription for defining cohomology with compact support from relative cohomology,
and Cheeger-Simons differential characters. 
In the following all manifolds will be assumed to be finite-dimensional, smooth, 
paracompact and Hausdorff. Sometimes we shall also consider manifolds with a (smooth) boundary.
For some of our constructions and results in Section \ref{secPontrDuality}
we demand further conditions (such as connectedness, orientability or existence of finite good covers),
which will be stated explicitly where needed.

\subsection{Smooth singular (co)homology and its relative version}\label{secHRelH}
Let $M$ be a manifold (possibly with boundary). 
We denote by $C_\sharp(M)$ the chain complex of smooth singular chains on $M$ with $\bbZ$-coefficients.
The boundary homomorphism is denoted by $\del_k: C_k(M) \to C_{k-1}(M)$ and we shall frequently omit
the subscript $_{k}$ as it will be clear from the context.
The Abelian groups of $k$-cycles and $k$-boundaries on $M$ are given by
$Z_k(M) := \ker \del_k$ and $B_k(M) := \im \del_{k+1}$, respectively. 
The smooth singular $k$-th homology group of $M$ is then defined as the quotient 
$\H_k(M):=Z_k(M) / B_k(M)$. 
\sk

For our purposes, we also have to consider homology on $M$ 
relative to a submanifold $S \subseteq M$ (possibly with boundary),
which is obtained by identifying all chains with support inside $S$ with $0$. 
More precisely, the inclusion $S \subseteq M$ allows us to consider
the chain complex $C_\sharp(S)$ of smooth singular chains on $S$ 
as a subcomplex of $C_\sharp(M)$. 
The complex of smooth singular chains on $M$ relative to $S$ is then defined
as the quotient $C_\sharp(M,S) := C_\sharp(M) / C_\sharp(S)$. 
The boundary homomorphism is denoted by $\del_k: C_k(M,S) \to C_{k-1}(M,S)$.
Notice that any $C_k(M,S)$ is a free Abelian group, even though
it is defined as a quotient.\footnote{
To prove this statement, we observe that the short exact sequence 
$0 \to C_k(S) \to C_k(M) \to C_k(M,S) \to 0$ can be split
by the unique homomorphism $C_k(M)\to C_k(S)$ 
which sends any $k$-simplex $\sigma \in C_k(M)$ with image contained in $S$ 
to the corresponding $\sigma\in C_k(S)$ and any other $k$-simplex to $0$.
It follows that $C_k(M,S)$ is a direct summand of the free Abelian group $C_k(M)$ 
and hence a free Abelian group as well.
}
Similarly to above, the Abelian groups of relative $k$-cycles $Z_k(M,S)$
and relative $k$-boundaries $B_k(M,S)$ are respectively given by 
the kernel and image of the boundary homomorphism in $C_\sharp(M,S)$.
The relative $k$-th homology group is defined as the quotient 
$\H_k(M,S):= Z_k(M,S) / B_k(M,S)$. 
\sk

Let $G$ be an arbitrary Abelian group. The cochain complex $C^\sharp(M;G)$ of 
$G$-valued smooth singular cochains on $M$ is defined by $C^k(M;G) := \Hom(C_k(M),G)$
together with the coboundary homomorphism $\cdel_k := \Hom(\del_{k+1},G): C^k(M;G) \to C^{k+1}(M;G)$. 
The Abelian groups of $G$-valued $k$-cocycles and $k$-coboundaries on $M$ are given by
$Z^k(M;G) := \ker \cdel_k$ and $B^k(M;G) := \im \cdel_{k-1}$, respectively. 
The $G$-valued smooth singular $k$-th cohomology group of $M$ is then defined as the quotient 
$\H^k(M;G) := Z^k(M;G) / B^k(M;G)$. The relative version $C^\sharp(M,S;G)$ of the cochain complex
is defined analogously.
We set $C^k(M,S;G) := \Hom(C_k(M,S),G)$ and
$\cdel_k := \Hom(\del_{k+1},G): C^k(M,S;G) \to C^{k+1}(M,S;G)$.
Moreover, $G$-valued relative $k$-cocycles are defined as
$Z^k(M,S;G) := \ker \cdel_k$,  $G$-valued relative $k$-coboundaries 
as $B^k(M,S;G):= \im \cdel_{k-1}$ and the 
$G$-valued relative smooth singular $k$-th cohomology group as
the quotient $\H^k(M,S;G) := Z^k(M,S;G) / B^k(M,S;G)$. 
\sk

Given any chain complex $C_\sharp$ of free Abelian groups 
and a short exact sequence $0 \to F \to G \to H \to 0$ of Abelian groups,
there exists a short exact sequence of cochain complexes
\begin{flalign}
\xymatrix{
0\ar[r] & \Hom(C_{\sharp},F) \ar[r]&\Hom(C_{\sharp},G) \ar[r]& \Hom(C_{\sharp},H) \ar[r] & 0~~.
}
\end{flalign}
By a well-known result in homological algebra, see e.g.\ \cite[Theorem 1.3.1]{Weibel},
the cohomology groups of these cochain complexes then fit into a long exact sequence. 
Applying this result to the chain complexes $C_\sharp(M)$ and 
$C_\sharp(M,S)$ we obtain the long exact sequences
\begin{subequations}
\begin{align}\label{eqLESH}
&\xymatrix@C=25pt{
\cdots \ar[r]
		&	\H^{k-1}(M;H) \ar[r]^-\beta
				&	\H^k(M;F) \ar[r]
						&	\H^k(M;G) \ar[r]	
								&	\H^k(M;H) \ar[r]^-\beta		&	\cdots~~,
}
\\[4pt]
\label{eqLESRelH}
&\xymatrix@C=15pt{
\cdots \ar[r]
		&	\H^{k-1}(M,S;H) \ar[r]^-\beta
				&	\H^k(M,S;F) \ar[r]
						&	\H^k(M,S;G) \ar[r]	
								&	\H^k(M,S;H) \ar[r]^-\beta	&	\cdots~~,
}
\end{align}
\end{subequations}
where $\beta$ denotes the connecting homomorphisms. 
\sk

In the following we will often refer to the functorial behavior of the absolute and relative
(co)homology groups. Let us introduce the relevant categories:
\begin{description}
\item[$\Man$:] 
The objects are manifolds and the morphisms are smooth maps. 

\item[$\Man_{m,\emb}$:] The subcategory of $\Man$ whose objects are $m$-dimensional manifolds
and morphisms are open embeddings.

\item[$\Pair$:]
The objects are pairs $(M,S)$ consisting of an object $M$ in $\Man$ and a submanifold $S \subseteq M$ 
(possibly with boundary) and the morphisms $f: (M,S) \to (M^\prime,S^\prime\,)$ 
are those morphisms $f: M \to M^\prime$ in $\Man$ which satisfy $f(S) \subseteq S^\prime$.

\item[$\PePair$:] The full subcategory of $\Pair$ whose objects 
$(M,S)$ are such that $S \subseteq M$ is a \emph{properly embedded} submanifold (possibly with boundary). 

\item[$\DSet$:]
The objects are directed sets and the morphisms are functions preserving the preorder relation. 
Alternatively, interpreting a directed set as a (small) category (with morphisms specified by the preorder relation),
we can interpret $\DSet$ as the full subcategory of the category 
of small categories $\Cat$ whose objects are directed sets. 

\item[$\Ab$:]
The objects are Abelian groups and the morphisms are group homomorphisms. 

\item[$\Ch(\Ab)$:] 
The objects are chain complexes of Abelian groups and the morphisms are chain maps. 
\end{description}

Interpreting cochain complexes $C^\sharp$ canonically as chain 
complexes via the reflection $C^k \to C^{-k}$, we observe that
absolute and relative smooth singular (co)chain complexes are functors
\begin{subequations}\label{eqChainFunctor}
\begin{align}
C_\sharp(-): \Man \longrightarrow \Ch(\Ab)~,			&&	C_\sharp(-): \Pair \longrightarrow \Ch(\Ab)~,			\\[4pt]
C^\sharp(-;G): \Man^\op \longrightarrow \Ch(\Ab)~,	&&	C^\sharp(-;G): \Pair^\op \longrightarrow \Ch(\Ab)~.
\end{align}
\end{subequations}
In fact, simplices in $M$ can be pushed forward along $f: M \to M^\prime$ and such a push-forward 
along a morphism $f: (M,S) \to (M^\prime,S^\prime\,)$ in $\Pair$ sends simplices in $S$ to simplices in $S^\prime$. 
Absolute and relative (co)homology inherit their functorial behavior from these functors, i.e.\
\begin{subequations}
\begin{align}
\H_k(-): \Man \longrightarrow \Ab~,					&&	\H_k(-): \Pair\longrightarrow \Ab~,					\\[4pt]
\H^k(-;G): \Man^\op \longrightarrow \Ab~,			&&	\H^k(-;G): \Pair^\op \longrightarrow \Ab~.
\end{align}
\end{subequations}

\subsection{Smooth singular cohomology with compact support}\label{secHc}
Following \cite[p.\ 242]{Hatcher}, we introduce smooth singular cohomology with compact support 
by means of a colimit prescription. Let 
\begin{equation}\label{eqCompSubsetFunctor}
\mcK: \Man \longrightarrow \DSet~ 
\end{equation}
be the functor which assigns to any manifold $M$ the directed set $\mcK_{M} := \{K\subseteq M \text{ compact}\}$ 
(with preorder relation given by subset inclusion) and to any smooth map
$f: M \to M^\prime$ the morphism $\mcK_{f}: \mcK_{M} \to \mcK_{M^\prime}\,,~ K \mapsto f(K)$ of directed sets.
Interpreting $\mcK_{M}$ as a (small) category, we obtain a functor 
\begin{equation}\label{eqCompSubsetToPairFunctor}
(M,M \setminus -): \mcK_M \longrightarrow \Pair^\op~,
\end{equation}
where the target category is $\Pair^\op$ because we take complements of subsets.
Given an Abelian group $G$, we can precompose the
relative smooth singular cohomology functor 
$\H^k(-;G): \Pair^\op \to \Ab$ with the functor \eqref{eqCompSubsetToPairFunctor}
and obtain $\H^k(M,M \setminus -;G): \mcK_M \to \Ab$. 
We then define the $G$-valued compactly supported smooth singular 
cohomology group of $M$ as the colimit of this functor, i.e.\
\begin{equation}\label{eqHcdef}
\H^k_{\rm c}(M;G) := \colim \big(\H^k(M,M \setminus -;G): \mcK_M \to \Ab\big)~. 
\end{equation}

\begin{rem}[Alternative definitions of compactly supported cohomology]\label{remAlternativeCoho}
Similarly to $\mcK_M$, we can define the directed set 
$\mcO_M^{\rm c} := \{ O \subseteq M \text{ open}~:~\overline O \in \mcK_M\,,~\del \overline O \text{ smooth}\}$ 
of relatively compact open subsets with smooth boundary.
As for $\mcK_M$, the preorder relation on $\mcO_M^{\rm c}$ is given by subset inclusion.
By construction, the complement $M \setminus O$ of any $O \in \mcO_M^{\rm c}$ is a properly embedded submanifold
and thus the assignment $O \mapsto (M,M\setminus O)$ defines a functor
\begin{equation}\label{eqCompSubsetToPairFunctor_2}
(M,M \setminus -): \mcO_M^{\rm c} \longrightarrow \Pair^\op_{\mathsf{pe}}~. 
\end{equation}
Composing this functor with the embedding $\Pair^\op_{\mathsf{pe}}\to \Pair^\op$
and the relative cohomology functor $\H^k(-;G): \Pair^\op \to \Ab$,
we obtain $\H^k(M,M \setminus -;G): \mcO_M^{\rm c} \to \Ab$.
We shall now show that the colimit of this functor provides an equivalent
definition of compactly supported cohomology:
Introducing the directed subset $\mcU_M := \mcK_M \cup \mcO_M^{\rm c}$ 
of the power set of $M$ yields another functor $\H^k(M,M \setminus -;G): \mcU_M \to \Ab$. 
Because both $\mcK_M$ and $\mcO_M^{\rm c}$ are cofinal in $\mcU_M$
we obtain the chain of isomorphisms
\begin{equation}
\H^k_{\rm c}(M;G) \simeq \colim \big(\H^k(M,M \setminus -;G): \mcU_M \to \Ab\big) 
\simeq \colim \big(\H^k(M,M \setminus -;G): \mcO_M^{\rm c} \to \Ab\big)~,
\end{equation}
which provides alternative definitions of $\H^k_{\rm c}(M;G) $.
\end{rem}

\paragraph*{Functoriality:}
We now prove that $G$-valued compactly supported smooth singular 
cohomology is a functor 
\begin{equation}
\H^k_{\rm c}(-;G): \Man_{m,\emb} \longrightarrow \Ab~.
\end{equation}
Given an open embedding $f: M \to M^\prime$ between $m$-dimensional manifolds, 
there is a factorization $f = i \circ g$ into the inclusion $i : f(M) \to M^\prime$
and a diffeomorphism $g: M \to f(M)$.
Since $f(M) \setminus K \subseteq M^\prime\setminus K$, for any $K\in \mcK_{f(M)}$,
the inclusion $i$ defines a  natural transformation 
$i^\ast: \H^k(M^\prime,M^\prime \setminus -;G) \circ \mcK_i \Rightarrow \H^k(f(M),f(M) \setminus -;G)$
between functors from $\mcK_{f(M)}$ to $\Ab$.
By a standard excision argument for cohomology,
we find that $i^\ast$ is a natural isomorphism and 
we denote its inverse by
\begin{equation}
(i^\ast)^{-1}: \H^k(f(M),f(M) \setminus -;G) \Longrightarrow \H^k(M^\prime,M^\prime \setminus -;G) \circ \mcK_i~.
\end{equation}
Furthermore, since $g$ is a diffeomorphism, we immediately get a natural isomorphism 
$g^\ast: \H^k(f(M),f(M) \setminus -;G) \circ \mcK_g \Rightarrow \H^k(M,M \setminus -;G)$
between functors from $\mcK_{M}$ to $\Ab$ and we denote its inverse by
\begin{equation}
(g^\ast)^{-1}: \H^k(M,M \setminus -;G) \Longrightarrow \H^k(f(M),f(M) \setminus -;G) \circ \mcK_g~.
\end{equation}
We have thereby shown that $f^\ast = g^\ast \circ (i^\ast\, \mcK_g): 
\H^k(M^\prime,M^\prime \setminus -;G) \circ \mcK_f \Rightarrow \H^k(M,M \setminus -;G)$ 
is a natural isomorphism with inverse denoted by
\begin{equation}
(f^\ast)^{-1} := ((i^\ast)^{-1}\, \mcK_g) \circ (g^\ast)^{-1}: 
\H^k(M,M \setminus -;G) \Longrightarrow \H^k(M^\prime,M^\prime \setminus -;G) \circ \mcK_f~.
\end{equation}
By the universal property of the colimit, the natural transformation $(f^\ast)^{-1}$
induces a unique homomorphism $f_\ast: \H^k_{\rm c}(M;G) \to \H^k_{\rm c}(M^\prime;G)$. 
It is easy to check that $\H^k_{\rm c}(-;G): \Man_{m,\emb} \to \Ab$ defined in this way is a functor.
\sk

We recall that assigning to diagrams in $\Ab$
over a {\em directed set} their colimits is an exact functor because $\Ab$ is a Grothendieck category. 
Applying this observation to \eqref{eqLESRelH}, with $S=M\setminus K$ running over $K\in \mcK_M$,
we obtain a long exact sequence
\begin{equation}\label{eqLESHc}
\xymatrix{
\cdots \ar[r]
		&	\H^{k-1}_{\rm c}(M;H) \ar[r]^-\beta
				&	\H^k_{\rm c}(M;F) \ar[r]
						&	\H^k_{\rm c}(M;G) \ar[r]	
								&	\H^k_{\rm c}(M;H) \ar[r]^-\beta	&	\cdots~
}
\end{equation}
for compactly supported cohomology.

\subsection{Cheeger-Simons differential characters}\label{subsecCS}
The starting point for our investigations is the graded commutative ring 
of Cheeger-Simons differential characters \cite{CS},
which was later recognized as a model for differential cohomology~\cite{BB, SS}. 
Different models have been developed in terms of de Rham-Federer currents \cite{HLZ}, 
differential cocycles \cite{HS} and smooth Deligne cohomology \cite{Brylinski}. 
In \cite{HL1} several models were shown to be canonically isomorphic 
and, using an axiomatic approach, it was later proven in \cite{SS, BB} 
that differential cohomology is uniquely determined up to unique natural equivalences.
See also \cite{Bu12} for a more abstract homotopy theoretic 
approach to (generalized) differential cohomology theories. 

\begin{defi}\label{defDiffChar}
A degree $k$ Cheeger-Simons differential character on a manifold $M$
is a homomorphism $h : Z_{k-1}(M)\to \bbT$ into the circle group $\bbT := \bbR/\bbZ$
for which there exists a differential form $\omega_h\in\Omega^k(M)$ such that
\begin{flalign}\label{eqn:curvproperty}
h(\del \gamma) = \int_\gamma\, \omega_h \mod \bbZ~,\qquad \forall \gamma \in C_k(M)~.
\end{flalign}
We denote the Abelian group of Cheeger-Simons differential characters by $\dH^k(M;\bbZ)$. 
\end{defi}

Given any $h \in \dH^k(M;\bbZ)$, it is easy to show that
$\omega_h \in \Omega^k(M)$ is uniquely specified by 
\eqref{eqn:curvproperty} and that it has integral periods. 
(Note that, on account of a rescaling argument, differential forms must vanish 
if they take values in a proper subring of the reals upon integration over any chain \cite{CS}. 
In particular, differential forms having integral periods must be closed.) 
Introducing the Abelian group of $k$-forms with integral periods as 
\begin{equation}\label{eqZForms}
\OmegaZ^k(M) 
:= \left\{\omega \in \Omega^k(M)~:~ \int_z\, \omega \in \bbZ\,,\; \forall z \in Z_k(M)\right\} 
\subseteq \Omega_\dd^k(M)~, 
\end{equation}
where the subscript $_\dd$ denotes closed forms, 
we obtain the curvature homomorphism
\begin{equation}
\cu: \dH^k(M;\bbZ) \longrightarrow \OmegaZ^k(M)~,\qquad h \longmapsto \omega_h~.
\end{equation}
Any $u \in \H^{k-1}(M;\bbT)$ may be interpreted 
as an element of $\dH^k(M;\bbZ)$ via the homomorphism
\begin{equation}\label{eqInclFlat}
\kappa: \H^{k-1}(M;\bbT) \longrightarrow \dH^k(M;\bbZ)
\end{equation}
called the ``inclusion of flat fields''. It is constructed as follows:
Since $\bbT$ is divisible, the universal coefficient theorem for cohomology implies
that there exists a natural isomorphism $\H^{k-1}(M;\bbT) \simeq \Hom(\H_{k-1}(M),\bbT)$. 
Given any $u \in \H^{k-1}(M;\bbT)$, via this isomorphism we regard it
as a homomorphism $u: \H_{k-1}(M) \to \bbT$ and define $\kappa\, u : Z_{k-1}(M)\to \bbT$
by precomposition with the quotient map $Z_{k-1}(M) \to \H_{k-1}(M)$.
By definition, the curvature of $\kappa\, u$ is $0$.
\sk

Because $Z_{k-1}(M)$ is a free Abelian group, 
we can lift any $h \in \dH^k(M;\bbZ)\subseteq \Hom(Z_{k-1}(M),\bbT)$ 
to $\tilde h \in \Hom(Z_{k-1}(M),\bbR)$ along the quotient $\bbR \to \bbT$. 
As a consequence of \eqref{eqn:curvproperty}, the cochain 
$\int_\cdot\, \omega_h - \tilde h \circ \del\in C^k(M;\bbR)$ 
factors through the inclusion $\bbZ \to \bbR$
and hence there exists a unique integral cochain $c_h \in C^k(M;\bbZ)$ satisfying 
$\tilde h \circ \del = \int_\cdot\, \omega_h - c_h$. 
One easily checks that $\cdel c_h = 0$ and that the 
cohomology class $[c_h] \in \H^k(M;\bbZ)$ is uniquely determined by $h$. 
This defines the characteristic class homomorphism
\begin{equation}
\ch: \dH^k(M;\bbZ) \longrightarrow \H^k(M;\bbZ)~, \qquad h \longmapsto [c_h]~.
\end{equation}
Any differential form $A \in \Omega^{k-1}(M)$
defines a differential character $h_A \in \dH^k(M;\bbZ)$ by setting
\begin{equation}
h_A(z) = \int_z\, A \mod \bbZ~,
\end{equation}
for all $z \in Z_{k-1}(M)$. By Stokes' theorem 
we observe that $h_A(\del \gamma) = \int_\gamma \, \dd A \mod \bbZ$, for all $\gamma \in C_k(M)$,
and hence that the curvature of $h_A$ is $\dd A$.
We further observe that $h_A$ has trivial characteristic class because
$\int_\cdot\, A \in \Hom(Z_{k-1},\bbR)$ is a lift of $h_A$ and that
$h_A$ is trivial for $A \in \OmegaZ^{k-1}(M)$. 
This defines the topological trivialization homomorphism
\begin{equation}\label{eqRelTopTriv}
\iota: \frac{\Omega^{k-1}(M)}{\OmegaZ^{k-1}(M)} \longrightarrow \dH^k(M;\bbZ)~, \qquad [A] \longmapsto h_A~.
\end{equation}

It is shown in \cite{CS, HLZ, SS, BB} that the group of
Cheeger-Simons differential characters 
fits into the commutative diagram 
\begin{equation}\label{eqDiffCharDia}
\xymatrix{
		&	0 \ar[d]	&	0 \ar[d]	&	0\ar[d]																\\
0 \ar[r]	&	\frac{\H^{k-1}(M;\bbR)}{\Hf^{k-1}(M;\bbZ)} \ar[r] \ar[d]
					&	\frac{\Omega^{k-1}(M)}{\OmegaZ^{k-1}(M)} \ar[r]^-\dd \ar[d]_-\iota
								&	\dd\Omega^{k-1}(M) \ar[r] \ar[d]							&	0	\\
0 \ar[r]	&	\H^{k-1}(M;\bbT) \ar[r]^-\kappa \ar[d]
					&	\dH^k(M;\bbZ) \ar[r]^-\cu \ar[d]_-\ch
								&	\OmegaZ^k(M) \ar[r] \ar[d]									&	0	\\
0 \ar[r]	&	\Ht^k(M;\bbZ) \ar[r] \ar[d]
					&	\H^k(M;\bbZ) \ar[r] \ar[d]
								&	\Hf^k(M;\bbZ) \ar[r] \ar[d] 									&	0	\\
		&	0		&	0		&	0
}
\end{equation}
with all rows and columns short exact sequences.
The short exact sequences in the left column and in the bottom row 
are obtained from the natural long exact sequence \eqref{eqLESH} for cohomology 
associated to the coefficient sequence $0 \to \bbZ \to \bbR \to \bbT \to 0$. 
In the top row and in the right column, one uses the natural isomorphism 
$\OmegaZ^k(M) / \dd \Omega^{k-1}(M) \simeq \Hf^k(M;\bbZ)$ which is a by-product of de~Rham's theorem.

\paragraph*{Functoriality:}
The assignment of the Abelian group $\dH^k(M;\bbZ)$ to each manifold $M$ defines a functor 
\begin{equation}\label{eqDiffCharFunctor}
\dH^k(-;\bbZ): \Man^\op \longrightarrow \Ab~.
\end{equation}
In particular, any smooth map $f: M \to M^\prime$ 
induces a pull-back $f^\ast: \dH^k(M^\prime;\bbZ) \to \dH^k(M;\bbZ)$ 
by dualizing the push-forward $f_\ast: Z_{k-1}(M) \to Z_{k-1}(M^\prime\,)$ for cycles. 
For any $h \in \dH^k(M^\prime;\bbZ)$, the homomorphism $f^\ast h : Z_{k-1}(M)\to\bbT$ fulfills 
the condition \eqref{eqn:curvproperty} with $\omega_{f^\ast h} = f^\ast \omega_h$. 
Moreover, all homomorphisms in diagram \eqref{eqDiffCharDia} are
natural transformations between functors from $\Man^\op$ to $\Ab$. In particular, we have
the natural transformations
\begin{subequations}\label{eqNatHomDiffChar}
\begin{align}
\cu: \dH^k(-;\bbZ) \Longrightarrow \OmegaZ^k(-)~,		&&	\ch: \dH^k(-;\bbZ)\Longrightarrow  \H^k(-;\bbZ)~,			\\[4pt]
\iota: \frac{\Omega^{k-1}(-)}{\OmegaZ^{k-1}(-)} \Longrightarrow  \dH^k(-;\bbZ)~,
											&&	\kappa: \H^{k-1}(-;\bbT) \Longrightarrow  \dH^k(-;\bbZ)~.
\end{align}
\end{subequations}

\begin{ex}\label{ex:linebundles}
The following examples are explained in detail in \cite[Part I, Section 5.1, Examples 18, 19 and 20]{BB}:
For $k=1$, the Abelian group of differential characters $\dH^1(M;\bbZ)$
is canonically isomorphic to the Abelian group of circle-valued functions $C^\infty(M,\bbT)$; 
for any $h\in C^\infty(M,\bbT)$, the curvature is $\cu(h) = \dd\log(h)$ and the
characteristic class is $\ch(h) = h^*[\bbT]$, where $[\bbT] \in \H^1(\bbT;\bbZ)$ denotes the fundamental class.
For $k=2$, the Abelian group of differential characters $\dH^2(M;\bbZ)$ describes
isomorphism classes of  Hermitean line bundles with  connection on $M$.
For $k=3$, the Abelian group of differential characters $\dH^3(M;\bbZ)$ describes
isomorphism classes of Abelian gerbes with connection on $M$.
\end{ex}

\begin{ex}\label{ex:CW}
Another source of examples of differential characters is provided 
by a refinement of classical Chern-Weil theory \cite{CS}:
Let $G$ be a Lie group with finitely many connected components and $P \to M$ 
a principal $G$-bundle with connection $\theta$.
Denote by $F_\theta$ the curvature form of $\theta$.
Associated to an invariant polynomial $\lambda: \mathfrak g^k \to \bbR$ on the 
Lie algebra and a corresponding universal characteristic class $u \in \H^{2k}(BG;\bbZ)$ 
is a differential character in $\dH^{2k}(M;\bbZ)$ with characteristic class $u(P)$ 
and curvature the Chern-Weil form $\lambda(F_\theta)$. 
The differential character is uniquely determined by these properties and the 
requirement of being natural with respect to connection preserving bundle morphisms.
The construction is reviewed in detail in \cite{B14} and also refined to relative 
differential characters by taking into account the Chern-Simons 
form of $(P,\theta)$ associated with $\lambda$.  
\end{ex}

\paragraph*{Ring structure:}
The Abelian groups $\dH^\sharp(M;\bbZ)$ can be endowed with a 
natural graded commutative ring structure
\begin{equation}\label{eqDiffCharRing0}
\cdot \,:\, \dH^k(M;\bbZ) \times \dH^l(M;\bbZ)	 \longrightarrow \dH^{k + l}(M;\bbZ)~, \qquad	(h,h'\,) \longmapsto h \cdot h'\,~.
\end{equation}
Following \cite{CS}, the construction is as follows:
Choose a family of natural cochain homotopies
$B: \Omega^k(M) \times \Omega^l(M) \to C^{k+l-1}(M;\bbR)$,
for all $k,l\geq 0$, between the wedge product for differential forms 
and the cup product for the corresponding cochains 
such that $B(\omega,\theta) = (-1)^{k \,l}\, B(\theta,\omega)$, 
for all $\omega \in \Omega^k(M)$ and $\theta \in \Omega^l(M)$. 
An example of such a family of cochain homotopies is given in \cite[Section 1]{CS} 
by means of iterated subdivisions, and the chain homotopy between subdivision and identity. 
Two different choices of $B$ turn out to be naturally cochain homotopic, 
see \cite[Section 3.2]{HS}. 
Given $h \in \dH^k(M;\bbZ)$ and $h'\, \in \dH^l(M;\bbZ)$, 
one defines their product $h \cdot h'\, \in \dH^{k+l}(M;\bbZ)$ 
as the homomorphism $Z_{k+l-1}(M)\to \bbT$ given by
\begin{subequations}\label{eqDiffCharRing}
\begin{align}
h \cdot h'\,		&	:= \tilde h \smile \int_\cdot\, \cu\, h'\, + (-1)^k \, c_h \smile \tilde h'\, + B(\cu\, h,\cu\, h'\,) \mod \bbZ	\\[4pt]
			&	\hphantom{:}= \tilde h \smile c_{h'}\, + (-1)^k \, \int_\cdot\, \cu\, h \smile \tilde h'\, + B(\cu\, h,\cu\, h'\,)	
					\mod \bbZ~,
\end{align}
\end{subequations}
where $\tilde h \in C^{k-1}(M;\bbR)$ and $\tilde h'\, \in C^{l-1}(M;\bbR)$ 
lift and extend $h$ and $h'\,$, respectively, while $c_h \in C^k(M;\bbZ)$ and $c_{h'}\, \in C^l(M;\bbZ)$ 
are cocycles such that $\cdel \tilde h = \int_\cdot\, \cu\, h - c_h$ and $\cdel \tilde h'\, = \int_\cdot\, \cu\, h'\, - c_{h'}$.
A proof that the map \eqref{eqDiffCharRing0} specified by \eqref{eqDiffCharRing} 
defines an associative and graded-commutative ring structure 
can be found in \cite[Theorem 1.11]{CS} or in \cite[Part II, Section 4.1.1]{BB}. 
Since both the cup product $\smile$ and the cochain homotopy $B$ are natural, 
the expression \eqref{eqDiffCharRing} defines a natural ring structure, i.e.\ 
the pull-back of differential characters along a smooth map is a ring homomorphism. 
The four natural transformations displayed in \eqref{eqNatHomDiffChar}
are compatible with the ring structure, i.e.\
\begin{subequations}
\begin{align}
\cu(h \cdot h'\,)		&	= \cu\, h \wedge \cu\, h'\,~,	&	\ch(h \cdot h'\,)		&	= \ch\, h \smile \ch\, h'\,~,\\[4pt]
\iota\, [A] \cdot h'\,	&	= \iota([A \wedge \cu\, h'\,])~,	&	\kappa\, u \cdot h'\,	&	= \kappa(u \smile \ch\, h'\,)~,
\end{align}
\end{subequations}
for all $h \in \dH^k(M;\bbZ)$, $h'\, \in \dH^l(M;\bbZ)$, 
$[A] \in \Omega^{k-1}(M) / \OmegaZ^{k-1}(M)$ and $u \in \H^{k-1}(M;\bbT)$.

\begin{ex}
The ring structure on $\dH^\sharp(M;\bbZ)$ provides 
a construction of an isomorphism class of  Hermitean line bundles with connection on $M$ 
out of two  circle-valued functions $h_1, h_2 \in C^\infty(M,\bbT)$.
This bundle can be described explicitly as the pull-back along the product map 
$(h_1, h_2): M \to \bbT^2$ of a universal line bundle with connection on the $2$-torus, 
called the \emph{Poincar\'e bundle}, see \cite{Bu12} and \cite{BB} for further details.
\end{ex}

\begin{ex}
The construction of differential characters from classical Chern-Weil theory in Example~\ref{ex:CW} 
is multiplicative: Given two invariant polynomials $\lambda_1,\lambda_2$ 
on $\mathfrak g$ and corresponding universal characteristic classes 
$u_1,u_2 \in \H^\sharp(BG;\bbZ)$, the differential character associated with 
$\lambda_1 \cdot \lambda_2$ coincides with the product of the differential characters 
from Example~\ref{ex:CW}. Its characteristic class is the cup product $u_1(P) \smile u_2(P)$, 
while its curvature is the wedge product $\lambda_1(F_\theta) \wedge \lambda_2(F_\theta)$ 
of the corresponding Chern-Weil forms. 
\end{ex}


\section{Differential characters on relative cycles}\label{secRelDiffChar}
In this section we review a version of relative differential cohomology 
which is used later to introduce compactly supported differential cohomology.
Recall that there are two different ways to define the smooth singular cohomology 
of a manifold $M$ relative to a submanifold $S \subseteq M$ (possibly with boundary): 
The first option is as the cohomology of the mapping cone complex of the inclusion $i_S:S \hookrightarrow M$ 
and the second option is as the cohomology of the cochain complex $C^\sharp(M,S;G) := \Hom(C_\sharp(M,S);G)$,
where $C_\sharp(M,S) := C_\sharp(M)/C_\sharp(S)$ is the quotient complex.
The latter version was discussed in more detail in Subsection \ref{secHRelH}.
A similar point of view can be taken for relative de Rham cohomology: It may be defined 
as the cohomology of the mapping cone complex $\Omega^\sharp(i_S)$ of the inclusion $i_S:S \hookrightarrow M$ as in~\cite[p.\ 78]{BT} 
or as the cohomology of the subcomplex $\Omega^\sharp(M,S)$ 
of forms vanishing on $S$ as in \cite[Chapter XII]{G}.
For a properly embedded submanifold $S \subseteq M$  (possibly with boundary), 
taking into account the long exact cohomology sequences 
arising from the relative/absolute exact sequences for $\Omega^\sharp(i_S)$ and for $\Omega^\sharp(M,S)$,
one concludes that both approaches give the same cohomology groups, although the complexes are different; explicitly, a five lemma argument shows that $\H^\sharp_\dR(M,S) \to \H^\sharp_\dR(i_S)$, 
$[\omega] \mapsto [\omega,0]$ is an isomorphism. 
\sk

Since differential cohomology is a refinement of smooth singular cohomology by differential forms, 
the question arises whether to refine the relative cohomology $\H^\sharp(M,S;\bbZ)$ by the 
mapping cone de Rham complex $\Omega^\sharp(i_S)$ or by the relative de Rham complex $\Omega^\sharp(M,S)$. 
Differential characters based on the mapping cone complex were first introduced in \cite{BrT} 
and they were called \emph{relative differential characters}; 
for a comparison with relative Deligne cohomology see also \cite{FeR}. 
Differential characters on relative cycles were first introduced in 
\cite{BB}\footnote{In this reference it is assumed that $S \subseteq M$
is a closed submanifold, but the constructions and results directly generalize
to the case of properly embedded submanifolds.} and they 
were called \emph{parallel relative differential characters}.
These two versions of relative differential cohomology fit into diagrams 
similar to \eqref{eqDiffCharDia}, provided that one considers properly embedded submanifolds $S\subseteq M$.
They also fit into long exact sequences relating absolute and relative differential cohomology groups, 
see \cite{BB} and below. See \cite{BB} also for a comparison 
between relative and parallel relative differential characters.

\subsection{Definition and first properties}\label{subRelDiffCharProp}
Let us begin by fixing 
our notation for relative de~Rham cohomology. Let $M$ be a manifold 
and $S\subseteq M$ a submanifold (possibly with boundary).
We denote by
\begin{equation}\label{eqRelForms}
\Omega^k(M,S)				:= \{\omega \in \Omega^k(M)\,:\; \omega\vert_S = 0\}
\end{equation}
the Abelian group of $k$-forms vanishing on $S\subseteq M$ and by
\begin{equation}\label{eqRelFormsZ}
\Omega^k_\bbZ(M,S)		:= \Big\{\omega \in \Omega^k(M,S)\,:\; 
									\int_z\, \omega \in \bbZ\,,\; \forall z \in Z_k(M,S)\Big\}
\subseteq \Omega_\dd^k(M,S)
\end{equation}
its subgroup of relative $k$-forms with integral periods 
(as in \eqref{eqZForms}, these are in particular closed).
The natural homomorphism $Z_k(M) \to Z_k(M,S)$ implies 
that $\Omega^k_\bbZ(M,S)$ is a subgroup of $\Omega^k_\bbZ(M)$. 
Since the exterior derivative $\dd$ preserves the 
subgroup $\Omega^\sharp(M,S)\subseteq \Omega^\sharp(M)$,
the relative de Rham complex $(\Omega^\sharp(M,S),\dd)$ 
is a subcomplex of the usual de Rham complex $(\Omega^\sharp(M),\dd)$. 
We denote the corresponding relative de Rham cohomology 
by $\H^\sharp_\dR(M,S)$. 

\begin{defi}\label{defRelDiffChar}
A degree $k$ differential character on relative cycles on a manifold $M$ with respect to $S\subseteq M$
is a homomorphism $h : Z_{k-1}(M,S)\to \bbT$
for which there exists a differential form $\omega_h\in\Omega^k(M)$ such that
\begin{flalign}\label{eqn:curvpropertyrel}
h(\del \gamma) = \int_\gamma\, \omega_h \mod \bbZ~,\qquad \forall \gamma \in C_k(M)~.
\end{flalign}
We denote the Abelian group of differential characters on relative cycles by $\dH^k(M,S;\bbZ)$. 
\end{defi}

As in the case for (absolute) differential characters,
the form $\omega_h$ is uniquely determined by $h\in  \dH^k(M,S;\bbZ)$.
Evaluating \eqref{eqn:curvpropertyrel} on $\gamma \in C_k(S)$ 
we obtain $\omega_h\vert_S = 0$ and evaluating
on  $\gamma =z \in Z_k(M,S)$ it follows that $\int_z \, \omega_h \in \bbZ$. 
This yields the curvature homomorphism for differential characters on relative cycles 
\begin{equation}
\cu: \dH^k(M,S;\bbZ) \longrightarrow \OmegaZ^k(M,S)~, \qquad h \longmapsto \omega_h~.
\end{equation}
Any $u \in \H^{k-1}(M,S;\bbT)$ may be interpreted 
as an element of $\dH^k(M,S;\bbZ)$ with  vanishing curvature.
The argument is exactly the same as for the absolute case (see the text following \eqref{eqInclFlat})
and we just have to replace all absolute (co)homology groups 
with their relative analogues. The corresponding homomorphism
\begin{equation}\label{eqRelInclFlat}
\kappa: \H^{k-1}(M,S;\bbT) \longrightarrow \dH^k(M,S;\bbZ)
\end{equation}
is called the ``inclusion of flat fields (on relative cycles)''.
\sk

Recalling that relative chains (and hence relative cycles) 
form a free Abelian group, as argued in Subsection \ref{secHRelH}, 
we can lift any  $h \in \dH^k(M,S;\bbZ)$ to an element
$\tilde h \in \Hom(Z_{k-1}(M,S),\bbR)$ along the quotient $\bbR \to \bbT$. 
As a consequence of \eqref{eqn:curvpropertyrel},
the relative cochain $\int_\cdot \, \omega_h - \tilde h \circ \del \in C^k(M,S;\bbR)$ 
factors through the inclusion $\bbZ \to \bbR$ and hence there exists a
unique integral relative cochain $c_h \in C^k(M,S;\bbZ)$ satisfying $\tilde h \circ \del = \int_\cdot\, \omega_h - c_h$. 
One easily checks that $\cdel c_h = 0$ 
and that the relative cohomology class $[c_h] \in \H^k(M,S;\bbZ)$ is uniquely determined by $h$. 
This defines the relative counterpart of the characteristic class homomorphism
\begin{equation}\label{eqRelCh}
\ch: \dH^k(M,S;\bbZ) \longrightarrow \H^k(M,S;\bbZ)~, \qquad h \longmapsto [c_h]~.
\end{equation}
Any relative differential form $A \in \Omega^{k-1}(M,S)$
defines a differential character on relative cycles $h_A \in \dH^k(M,S;\bbZ)$ 
by setting
\begin{equation}\label{eqn:hAdef}
h_A(z) = \int_z\, A \mod \bbZ~,
\end{equation}
for all $z \in Z_{k-1}(M,S)$. The curvature of $h_A$ is $\dd A$
and the relative characteristic class is $0$.
Notice further that for $A \in \OmegaZ^{k-1}(M,S)$
integration over relative $(k-1)$-cycles takes values in $\bbZ$, i.e.\ $h_A =0$. 
This defines the relative version of the topological trivialization homomorphism
\begin{equation}
\iota: \frac{\Omega^{k-1}(M,S)}{\OmegaZ^{k-1}(M,S)} \longrightarrow
\dH^k(M,S)~, \qquad [A] \longmapsto h_A~.
\end{equation}
\begin{theo}\label{thmRelDiffChar}
Let $M$ be a manifold and $S \subseteq M$ a submanifold (possibly with boundary). 
Then all squares in the diagram
\begin{equation}\label{eqRelDiffCharDia}
\xymatrix{
		&	0 \ar[d]																						\\
		&	\frac{\H^{k-1}(M,S;\bbR)}{\Hf^{k-1}(M,S;\bbZ)} \ar[d]
					&	\frac{\Omega^{k-1}(M,S)}{\OmegaZ^{k-1}(M,S)} \ar[r]^-\dd \ar[d]_-\iota
								&	\dd \Omega^{k-1}(M,S) \ar[d]										\\
0 \ar[r]	&	\H^{k-1}(M,S;\bbT) \ar[r]^-\kappa \ar[d]	
					&	\dH^k(M,S;\bbZ) \ar[r]^-\cu \ar[d]_-\ch
								&	\OmegaZ^k(M,S) \ar[r] \ar[d]						&	0			\\
0 \ar[r]	&	\Ht^k(M,S;\bbZ)	 \ar[r] \ar[d]
					&	\H^k(M,S;\bbZ) \ar[r]
								&	\Hf^k(M,S;\bbZ) \ar[r]								&	0			\\
		&	0
}
\end{equation}
commute. The left column, the middle row and the bottom row are short exact sequences. 
The middle and right columns form sequences starting with injections 
and the homomorphism in the top row is surjective. 
\end{theo}
\begin{proof}
As in the absolute case, the exact sequences in the left column and in the bottom row 
are obtained from the long exact sequence \eqref{eqLESRelH} for relative cohomology 
associated to the coefficient sequence $0 \to \bbZ \to \bbR \to \bbT \to 0$. 
In the right column, the homomorphism $\dd \Omega^{k-1}(M,S) \to \OmegaZ^k(M,S)$
is just the inclusion and the homomorphism $\OmegaZ^k(M,S) \to \Hf^k(M,S;\bbZ)$ 
is obtained by identifying $\Hf^k(M,S;\bbZ)$ with $\Hom(\H_k(M,S);\bbZ)$ 
and mapping $\omega \in \OmegaZ^k(M,S)$ 
to $\int_\cdot\, \omega \in \Hom(\H_k(M,S),\bbZ)$. 
The right column forms a sequence because of Stokes' theorem. In the top row,
$\dd: \Omega^{k-1}(M,S) / \OmegaZ^{k-1}(M,S) \to \dd \Omega^{k-1}(M,S)$ is surjective by definition. 
\sk

Commutativity of the top right square follows immediately from \eqref{eqn:hAdef}. 
To show commutativity of the bottom left square, 
observe that the composition of the left and bottom arrows 
is the connecting homomorphism $\beta: \H^{k-1}(M,S;\bbT) \to \H^k(M,S;\bbZ)$ in \eqref{eqLESRelH}. 
For $u \in \H^{k-1}(M,S;\bbT)$, $\beta\, u$ is defined by choosing a representative 
$\bar u \in Z^{k-1}(M,S;\bbT)$ of $u$, 
lifting $\bar u$ to $\tilde{\bar u} \in C^{k-1}(M,S;\bbR)$, taking the unique $c_u \in Z^k(M,S;\bbZ)$ 
such that $c_u = \cdel \tilde{\bar u}$ and setting $\beta\, u = [c_u]$. 
Assigning to $u$ the homomorphism $ \kappa\, u : Z_{k-1}(M,S)\to \bbT$ 
is by definition the same as restricting $\bar u$ to relative cycles. 
Hence the restriction to relative cycles of $\tilde {\bar u} \in C^{k-1}(M,S;\bbR)$ 
provides a lift $\widetilde{\kappa\, u} \in \Hom(Z_{k-1}(M,S),\bbR)$ of $\kappa\, u \in \Hom(Z_{k-1}(M,S),\bbT)$ along the 
quotient $\bbR \to \bbT$. 
Clearly $\widetilde{\kappa \, u} \circ \del = \tilde{\bar u} \circ \del$, 
therefore $\ch\, \kappa\, u = [c_u] = \beta\, u$ and the bottom left square is commutative as claimed. 
Let us  now show commutativity of the bottom right square. Given any $h \in \dH^k(M,S;\bbZ)$ 
and exploiting divisibility of $\bbT$, we choose an extension $\bar h \in C^{k-1}(M,S;\bbT)$ of $h$ 
and a lift $\tilde{\bar h} \in C^{k-1}(M,S;\bbR)$ of $\bar h$ along the quotient $\bbR \to \bbT$. 
By definition $\ch\, h = [c_h] \in \H^k(M,S;\bbZ)$, 
for $c_h \in Z^k(M,S;\bbZ)$ such that $\cdel \tilde{\bar h} = \int_\cdot\, \cu\, h - c_h$, 
i.e.\ the class in $\Hf^k(M,S;\bbZ)$ represented by $c_h$ 
is the same as the one represented by $\int_\cdot \, \cu\, h$. 
\sk

It is straightforward to prove that the middle column forms a sequence, 
i.e.\ $\ch \circ \iota = 0$. Furthermore, the first arrow is injective
by \eqref{eqn:hAdef} and the definition of $\OmegaZ^{k-1}(M,S)$, cf.\ \eqref{eqRelFormsZ}.
It remains to show that the middle row is a short exact sequence. 
First, let us notice that $\cu \circ \kappa = 0$ since $u \in \H^{k-1}(M,S;\bbT)$ 
vanishes when evaluated on relative boundaries. 
Furthermore, if $u \in \H^{k-1}(M,S;\bbT)$ is such that $\kappa\, u = 0$
then $u$ vanishes on all relative cycles, i.e.\ $u = 0$ and hence $\kappa$ is injective.
To show that $\cu$ is surjective, we exploit exactness of the bottom row.
Given any $\omega \in \OmegaZ^k(M,S)$, 
we find a preimage $[c] \in \H^k(M,S;\bbZ)$ of $\int_\cdot\, \omega \in \Hf^k(M,S;\bbZ)$. 
Hence there exists $\tilde h \in C^{k-1}(M,S;\bbR)$ 
such that $\int_\cdot\, \omega = c + \cdel \tilde h \in C^k(M,S;\bbR)$, 
where $c \in Z^k(M,S;\bbZ)$ denotes a representative of $[c]$. 
Let us denote by $h \in \Hom(Z_{k-1}(M,S),\bbT)$ the restriction of $\tilde h \mod \bbZ$. 
For each $\gamma \in C_k(M)$ we find that $h(\del \gamma) = \int_\gamma\, \omega \mod \bbZ$, 
which implies $h \in \dH^k(M,S;\bbZ)$ and $\cu\, h = \omega$ according to Definition \ref{defRelDiffChar}. 
This shows that $\cu$ is surjective and we are left with proving that $\ker \cu = \im \kappa$. 
Let $h \in \dH^k(M,S;\bbZ)$ be such that $\cu\, h = 0$. 
Then $h: Z_{k-1}(M,S) \to \bbT$ vanishes on $B_{k-1}(M,S)$ and it 
descends to $\underline h \in \Hom(\H_{k-1}(M,S),\bbT)$. 
Recalling that $\H^{k-1}(M,S;\bbT) \simeq \Hom(\H_{k-1}(M,S),\bbT)$, 
we find $u \in \H^{k-1}(M,S;\bbT)$ corresponding to $\underline h$.
The definition of $\kappa$ discussed above \eqref{eqRelInclFlat} then shows that $\kappa\, u = h$.
\end{proof}

\paragraph*{Functoriality:}
The assignment of the Abelian groups $\dH^k(M,S;\bbZ)$
to objects $(M,S)$ in the category $\Pair$ (see Subsection \ref{secHRelH})
defines a functor
\begin{equation}
\dH^k(-;\bbZ): \Pair^\op \longrightarrow \Ab~.
\end{equation}
Given any morphism $f: (M,S) \to (M^\prime,S^\prime\, )$ in $\Pair$,
we define for any $h\in \dH^k(M^\prime,S^\prime;\bbZ)$
the homomorphism $f^\ast h := h \circ f_\ast : Z_{k-1}(M,S)\to \bbT$ 
by exploiting functoriality of relative chains, see \eqref{eqChainFunctor}. 
Then $f^\ast h \in \dH^k(M,S;\bbZ)$ because of
$h(f_\ast \del \gamma) = \int_\gamma \, f^\ast \cu\, h$, 
for all $\gamma \in C_k(M)$. 
This also shows that the curvature $\cu$ is a natural transformation 
for differential characters on relative cycles. 
One can also easily show that $\kappa$, $\ch$ and $\iota$
are natural transformations for differential characters on relative cycles, i.e.\
\begin{subequations}
\begin{align}
\cu: \dH^k(-;\bbZ) \Longrightarrow \OmegaZ^k(-)~,		&&	\ch: \dH^k(-;\bbZ) \Longrightarrow \H^k(-;\bbZ)~,			\\[4pt]
\iota: \frac{\Omega^{k-1}(-)}{\OmegaZ^{k-1}(-)} \Longrightarrow \dH^k(-;\bbZ)~,
											&&	\kappa: \H^{k-1}(-;\bbT) \Longrightarrow \dH^k(-;\bbZ)~,
\end{align}
\end{subequations}
are natural transformations between functors from $\Pair^\op$ to $\Ab$.
Moreover, all the arrows in the diagram displayed in Theorem \ref{thmRelDiffChar} 
are (the components of) natural transformations,
which follows from naturality of the long exact sequence
\eqref{eqLESRelH}.

\paragraph*{The natural homomorphism $\boldsymbol I$:}
Since the quotient map $C_\sharp(M) \to C_\sharp(M,S)$ preserves 
the boundary homomorphisms $\del$, it maps cycles to relative cycles.
Precomposing differential characters on relative cycles with this quotient map
thus defines a natural homomorphism
\begin{equation}\label{eqI}
I:\dH^k(M,S;\bbZ) \longrightarrow \dH^k(M;\bbZ)
\end{equation}
that maps differential characters on relative cycles to Cheeger-Simons differential characters. 
Naturality of $I$ is expressed by commutativity of the diagram
\begin{equation}
\xymatrix{
\dH^k(M^\prime,S^\prime;\bbZ) \ar[r]^-I \ar[d]_-{f^\ast}	&	\dH^k(M^\prime;\bbZ) \ar[d]^-{f^\ast}	\\
\dH^k(M,S;\bbZ) \ar[r]_-I									&	\dH^k(M;\bbZ)
}
\end{equation}
for all morphisms $f: (M,S) \to (M^\prime,S^\prime\, )$ in $\Pair$,
which follows from the commutative diagram
\begin{equation}
\xymatrix{
Z_{k-1}(M) \ar[r] \ar[d]_{f_\ast}		&	Z_{k-1}(M,S) \ar[d]^-{f_\ast}	\\
Z_{k-1}(M^\prime\, ) \ar[r]			&	Z_{k-1}(M^\prime,S^\prime\, )
}
\end{equation}
for the quotient maps. In addition the natural diagrams
\begin{flalign}
\xymatrix{
\ar[d]_-{\cu}\dH^k(M,S;\bbZ) \ar[r]^-{I} & \dH^k(M;\bbZ) \ar[d]^-{\cu} & & \ar[d]_-{\ch}\dH^k(M,S;\bbZ) \ar[r]^-{I} & \dH^k(M;\bbZ) \ar[d]^-{\ch}\\
\Omega^k_\bbZ(M,S) \ar[r]&\Omega^k_\bbZ(M)  &&  \H^k(M,S;\bbZ) \ar[r] & \H^k(M;\bbZ)\\
 \ar[d]_-{\kappa} \H^{k-1}(M,S;\bbT) \ar[r]& \H^{k-1}(M;\bbT)\ar[d]^-{\kappa}   &&  \ar[d]_-{\iota}\frac{\Omega^{k-1}(M,S)}{\Omega^{k-1}_{\bbZ}(M,S)} \ar[r] & \frac{\Omega^{k-1}(M)}{\Omega^{k-1}_{\bbZ}(M)}\ar[d]^-{\iota}\\
\dH^k(M,S;\bbZ) \ar[r]_-{I} & \dH^k(M;\bbZ) && \dH^k(M,S;\bbZ) \ar[r]_-{I} & \dH^k(M;\bbZ)
}
\end{flalign}
commute, for all objects $(M,S)$ in $\Pair$.
The unlabeled arrows involving differential forms are induced by the inclusions $\Omega^p(M,S)\to \Omega^p(M)$
and the unlabeled arrows involving cohomology groups  are
the canonical homomorphisms from relative to absolute cohomology.

\begin{rem}\label{remRelNonSub}
In general, the homomorphism $I:\dH^k(M,S;\bbZ) \to \dH^k(M;\bbZ)$ fails to be injective. 
To illustrate this fact, consider the commutative diagram
\begin{equation}
\xymatrix{
0 \ar[r]	&	\H^{k-1}(M,S;\bbT)	\ar[r]^-\kappa \ar[d]
				&	\dH^k(M,S;\bbZ) \ar[r]^-\cu \ar[d]^-I
						&	\OmegaZ^k(M,S) \ar[d] \ar[r]			&	0	\\
0 \ar[r]	&	\H^{k-1}(M;\bbT) \ar[r]_-\kappa
				&	\dH^k(M;\bbZ) \ar[r]_-\cu
						&	\OmegaZ^k(M) \ar[r]					&	0
}
\end{equation}
with both rows being short exact sequences.
Since the right vertical arrow is injective (because it is a subset inclusion), 
the middle vertical arrow is injective if and only if so is the left one. We now construct examples of pairs $(M,S)$
for which $\H^{k-1}(M,S;\bbT) \to \H^{k-1}(M;\bbT)$ is not injective: Let $m \geq 2 $ and $k \in \{2,\ldots,m+1\}$. 
Consider $M = \bbR^m$ and $S = M \setminus (\bbR^{m-k+1} \times \bbB^{k-1})$, 
where $\bbB^p$ is a closed $p$-ball in $\bbR^p$. Observe that 
$M$ is homotopic to a point and $S$ is homotopic to the $(k-2)$-sphere $\bbS^{k-2}$.
Using the long exact sequence relating the relative cohomology of the 
pair $(M,S)$ to the cohomologies of $M$ and $S$, cf.\ \cite[p.\ 200]{Hatcher},
we obtain the exact sequence
\begin{equation}
\xymatrix{
\H^{k-2}(M;\bbT) \ar[r]
						&	\H^{k-2}(S;\bbT) \ar[r]
									&	\H^{k-1}(M,S;\bbT) \ar[r]
												&	\H^{k-1}(M;\bbT)~.
}
\end{equation}
Since $\H^{k-1}(M;\bbT)$ is trivial by construction, $\H^{k-1}(M,S;\bbT)$ can be computed 
as the quotient of $\H^{k-2}(S;\bbT)$ by the image of $\H^{k-2}(M;\bbT)$. 
Specifically, one finds that $\bbT\simeq \H^{k-1}(M,S;\bbT) \to \H^{k-1}(M;\bbT) \simeq 0$
is not injective. 
\end{rem}

\begin{ex}\label{ex:parallellinebundles}
As in Example~\ref{ex:linebundles}, the relative differential cohomology group 
$\dH^2(M,S;\bbZ)$ in degree $k=2$ has an immediate geometric interpretation. It is 
canonically isomorphic (by the holonomy map) to the group of isomorphism classes of triples $(L,\nabla,\sigma)$ 
consisting of a  Hermitean line bundle $L \to M$ with connection 
$\nabla$ and a $\nabla$-parallel section $\sigma:S \to L|_S$ of the restricted bundle
$L\vert_S \to S$. The homomorphism $I: \dH^2(M,S;\bbZ) \to \dH^2(M;\bbZ)$ is 
then induced by the forgetful map from triples $(L,\nabla,\sigma)$ to pairs $(L,\nabla)$ 
that ignores the section, see \cite{BB} for details. 
Since there may be inequivalent parallel sections on the same pair $(L,\nabla)$,
this gives a geometric interpretation of the non-injectivity of $I$, cf.\  Remark \ref{remRelNonSub}.
\end{ex}

\paragraph*{Properly embedded $\boldsymbol{S\subseteq M}$:}
In the following we shall specialize to the case where $S \subseteq M$ 
is a properly embedded submanifold (possibly with boundary).
In this case, differential forms on $S$ can be extended to differential forms on $M$, 
see e.g.\ \cite[Problem 10-9]{Lee}. 
In particular, we obtain the short exact sequence of de Rham complexes
\begin{equation}
\xymatrix{
0 \ar[r]	&	\Omega^\sharp(M,S) \ar[r]	&	\Omega^\sharp(M) \ar[r]	&	\Omega^\sharp(S) \ar[r]	&	0~.
}
\end{equation}
Regarding differential forms as cochains via integration 
over smooth singular chains, we obtain a commutative diagram of cochain 
complexes of Abelian groups 
\begin{equation}
\xymatrix{
0 \ar[r]	&	\Omega^\sharp(M,S) \ar[r] \ar[d]_{\int_\cdot}
					&	\Omega^\sharp(M) \ar[r] \ar[d]_{\int_\cdot}
								&	\Omega^\sharp(S) \ar[r] \ar[d]_{\int_\cdot}	&	0		\\
0 \ar[r]	&	C^\sharp(M,S;\bbR) \ar[r]
					&	C^\sharp(M;\bbR) \ar[r]
								&	C^\sharp(S;\bbR)\ar[r]						&	0
}
\end{equation}
and the corresponding commutative diagram of the long exact cohomology sequences 
\begin{equation}
\xymatrix@C=12pt{
\cdots \ar[r] &	\H^{k-1}_\dR(M) \ar[r] \ar[d]_\simeq
						&	\H^{k-1}_\dR(S) \ar[r] \ar[d]_\simeq
									&	\H^k_\dR(M,S) \ar[r] \ar[d]
												&	\H^k_\dR(M) \ar[r] \ar[d]_\simeq
															&	\H^k_\dR(S) \ar[r] \ar[d]_\simeq
																		&	\cdots						\\
\cdots \ar[r]	&	\H^{k-1}(M;\bbR) \ar[r]
						&	\H^{k-1}(S;\bbR) \ar[r]
									&	\H^k(M,S;\bbR) \ar[r]
												&	\H^k(M;\bbR) \ar[r]
															&	\H^k(S;\bbR) \ar[r]
																		&	\cdots\,
}
\end{equation}
By de Rham's theorem,\footnote{The de Rham theorem also holds for manifolds with boundary; 
the well-known proof due to A.~Weil using \v{C}ech-de Rham and \v{C}ech-singular double complexes 
can easily be adapted to the case of manifolds with boundary. 
Alternatively, one may also argue with homotopy invariance of de Rham cohomology 
using the fact that the inclusion $M \backslash \partial M \hookrightarrow M$ is a homotopy equivalence.} 
the vertical arrows between absolute cohomology groups are isomorphisms
and hence by the five lemma we conclude that also $\H^k_\dR(M,S) \to \H^k(M,S;\bbR)$ is an isomorphism.
This provides us with a relative version of de Rham's theorem for the case of $S\subseteq M$ being
properly embedded. Using this result we can refine the statement 
of Theorem~\ref{thmRelDiffChar} to obtain the full commutative diagram for relative differential cohomology.
\begin{theo}\label{thmRelDiffCharRefined}
Let $M$ be a manifold and $S \subseteq M$ a properly embedded submanifold (possibly with boundary). 
Then the diagram
\begin{equation}\label{eqRelDiffCharDiaAlt}
\xymatrix{
		&	0 \ar[d]	&	0 \ar[d]	&	0 \ar[d]																\\
0 \ar[r]	&	\frac{\H^{k-1}(M,S;\bbR)}{\Hf^{k-1}(M,S;\bbZ)} \ar[r] \ar[d]
					&	\frac{\Omega^{k-1}(M,S)}{\OmegaZ^{k-1}(M,S)} \ar[r]^-\dd \ar[d]_-\iota
								&	\dd \Omega^{k-1}(M,S) \ar[r] \ar[d]					&	0			\\
0 \ar[r]	&	\H^{k-1}(M,S;\bbT) \ar[r]^-\kappa \ar[d]	
					&	\dH^k(M,S;\bbZ) \ar[r]^-\cu \ar[d]_-\ch
								&	\OmegaZ^k(M,S) \ar[r] \ar[d]						&	0			\\
0 \ar[r]	&	\Ht^k(M,S;\bbZ)	 \ar[r] \ar[d]
					&	\H^k(M,S;\bbZ) \ar[r] \ar[d]
								&	\Hf^k(M,S;\bbZ) \ar[r] \ar[d]							&	0			\\
		&	0		&	0		&	0
}
\end{equation}
commutes and its rows and columns are short exact sequences.
\end{theo}
\begin{rem}
This theorem is proven in \cite{BB} for closed submanifolds $S \subseteq M$.
The proof carries over to the more general setting of properly embedded submanifolds (possibly with boundary).
For completeness, we briefly outline the proof.
\end{rem}
\begin{proof}
We first show that the right column is a short exact sequence. 
For this, we only have to prove that the morphism 
$\OmegaZ^k(M,S) \to \Hf^k(M,S;\bbZ)$ is surjective and that its kernel coincides 
with $\dd \Omega^{k-1}(M,S)$. Both statements follow from de Rham's 
theorem for relative cohomology
by taking into account the inclusion $\Hf^k(M,S;\bbZ) \subseteq \H^k(M,S;\bbR)$. 
\sk

Using this result we can also complete the diagram in Theorem \ref{thmRelDiffChar} by defining the missing horizontal
arrow in the top row: With $\Omega^{p}_\dd(M,S)$ denoting the closed relative $p$-forms we have
\begin{equation}
\frac{\H^{k-1}(M,S;\bbR)}{\Hf^{k-1}(M,S;\bbZ)} \simeq \frac{\Omega^{k-1}_\dd(M,S)}{\OmegaZ^{k-1}(M,S)} 
\subseteq \frac{\Omega^{k-1}(M,S)}{\OmegaZ^{k-1}(M,S)}~. 
\end{equation}
This map is injective and its image agrees with the 
kernel of $\dd: \Omega^{k-1}(M,S) / \OmegaZ^{k-1}(M,S) \to \dd \Omega^{k-1}(M,S)$.
Hence the top row is a short exact sequence. 
To show that the top left square commutes, it is sufficient to represent 
cohomology classes in $\H^{k-1}(M,S;\bbR)$ 
by means of closed $(k-1)$-forms according to the isomorphism displayed above. 
Exactness of the middle column follows from exactness of the 
other sequences. 
\end{proof}

\begin{rem}
In the present case of properly embedded submanifolds $S\subseteq M$ 
we can strengthen Remark~\ref{remRelNonSub} on the
non-injectivity of the homomorphism $I:\dH^k(M,S;\bbZ) \to \dH^k(M;\bbZ)$ 
by fitting it into a long exact sequence connecting absolute 
and relative (differential) cohomology groups.
By similar arguments as in \cite[Part II, Section 3.3.4]{BB}, 
there exists a long exact sequence
\begin{equation}\label{eq:long_ex_sequ_par}
\xymatrix@C=44pt{
\cdots \ar[r]
& \H^{k-2}(M,S;\bbT) \ar[r] & \H^{k-2}(M;\bbT) \ar[r] & \H^{k-2}(S;\bbT) \ar`[r]`d[lll]`[llld]`[llldr][lld]^(0.3){\kappa \circ \beta} &  \\  
& \dH^k(M,S;\bbZ) \ar[r]^-{I} & \dH^k(M;\bbZ) \ar^{i_S^*}[r] & \dH^k(S;\bbZ) \ar`[r]`d[lll]`[llld]`[llldr][lld]^(0.3){\beta \circ \ch} & \\
& \H^{k+1}(M,S;\bbZ) \ar[r] & \H^{k+1}(M;\bbZ) \ar[r] & \H^{k+1}(S;\bbZ) \ar[r] &  \cdots
} 
\end{equation}
of Abelian groups, where $\beta: \H^k(S;G) \to \H^{k+1}(M,S;G)$ denotes the connecting 
homomorphism (for $G = \bbT$ or $G=\bbZ$).
From \eqref{eq:long_ex_sequ_par} it follows that $h \in \dH^k(M;\bbZ)$ descends to a differential character 
on relative cycles if and only if it vanishes upon pull-back to $S$.  
\end{rem}

\subsection{Excision theorem}\label{subRelDiffCharNat}
We now prove a version of the excision theorem for differential characters
on relative cycles. This result will be used later in Section \ref{secCDiffChar} 
to define the push-forward of compactly supported differential characters and 
hence to understand their functorial behavior.

\begin{theo}\label{thmDiffCharExcision}
Let $M$ be a manifold. Consider $O \subseteq M$ open and 
$C \subseteq M$ closed such that $C \subseteq O$. 
Then the morphism $i : (O,O \setminus C) \to (M,M \setminus C)$
 in $\Pair$ induces an isomorphism 
\begin{equation}
i^\ast : \dH^k(M,M \setminus C;\bbZ) \longrightarrow \dH^k(O,O \setminus C;\bbZ)~.
\end{equation}
\end{theo}
\begin{proof}
Consider the central row of diagram \eqref{eqRelDiffCharDia} 
and recall that it is a natural short exact sequence. 
Hence the morphism $i: (O,O \setminus C) \to (M,M \setminus C)$ in $\Pair$ 
induces the commutative diagram 
\begin{equation}
\xymatrix{
0 \ar[r]	&	\H^{k-1}(M,M \setminus C;\bbT) \ar[r]^-\kappa \ar[d]_-{i^\ast}
					&	\dH^k(M,M \setminus C;\bbZ) \ar[r]^-\cu \ar[d]_-{i^\ast}
								&	\OmegaZ^k(M,M \setminus C) \ar[r] \ar[d]_-{i^\ast}		&	0	\\
0 \ar[r]	&	\H^{k-1}(O,O \setminus C;\bbT) \ar[r]_-\kappa
					&	\dH^k(O,O \setminus C;\bbZ) \ar[r]_-\cu
								&	\OmegaZ^k(O,O \setminus C) \ar[r]					&	0
}
\end{equation}
Excision for ordinary cohomology implies that the left vertical arrow is an isomorphism.
By the five lemma, it is sufficient to show that also the right vertical arrow is an isomorphism 
in order to complete the proof. 
\sk

We will now construct an inverse of the homomorphism 
$i^\ast: \OmegaZ^k(M,M \setminus C) \to \OmegaZ^k(O,O \setminus C)$. 
First, notice that $i^\ast: \Omega^k(M,M \setminus C) \to \Omega^k(O,O \setminus C)$ 
(without the restriction to integral periods) is an isomorphism because 
forms on $O$ that vanish on $O\setminus C$ can be extended by zero.
It remains to prove that the extension by zero homomorphism 
$(i^\ast)^{-1}: \Omega^k(O,O \setminus C) \to \Omega^k(M,M \setminus C)$ 
preserves integral periods. Let $\omega \in \OmegaZ^k(O,O \setminus C)$ and $z \in Z_k(M, M \setminus C)$. Choosing a representative $\tilde z \in C_k(M)$ of $z$, 
we find $\del \tilde z \in C_{k-1}(M \setminus C)$. 
Because $\{O,M \setminus C\}$ is an open cover of $M$, there exists an integer $j \geq 0$ 
such that the $j$-th iterated subdivision $S^j \tilde z$ is a combination 
of simplices which are supported either in $O$ or in $M \setminus C$. 
Let $a \in C_k(O)$ denote the combination of those simplices in 
$S^j \tilde z$ whose support intersects $C$. 
By construction $b := S^j \tilde z - a \in C_k(M \setminus C)$ 
and $\del a = S^j \del \tilde z - \del b \in C_{k-1}(O \setminus C)$.
In particular, $a$ represents an element of $Z_k(O,O \setminus C)$. 
Recalling that there exists a natural chain homotopy $D_j: C_p(M) \to C_{p+1}(M)$ 
between identity and $j$-th iterated subdivision, 
we conclude that $\tilde z = a + b - D_j \del \tilde z - \del D_j \tilde z$. 
By naturality, the chain homotopy $D_j$ preserves the supports of chains, 
in particular $D_j \del \tilde z \in C_k(M \setminus C)$. 
Since $\omega$ is closed, so is its extension by zero $(i^\ast)^{-1} \omega$.
This implies that
\begin{equation}
\int_z\, (i^\ast)^{-1} \omega
= \int_a\, \omega + \int_{b - D_j \del \tilde z} \, (i^\ast)^{-1} \omega - \int_{D_j \tilde z} \, \dd (i^\ast)^{-1} \omega
= \int_a\, \omega \in \bbZ~,
\end{equation}
where we have also used Stokes' theorem.
\end{proof}

\subsection{Module structure}\label{subRelDiffCharModule}
We show that relative differential cohomology $\dH^\sharp(M,S;\bbZ)$ is a module 
over the differential cohomology ring $\dH^\sharp(M;\bbZ)$, see Subsection \ref{subsecCS}.
Let $(M,S)$ be an object in $\Pair$. In the following we shall explain how
\eqref{eqDiffCharRing} may be used to define a bihomomorphism
\begin{equation}\label{eqRelDiffCharModule}
\cdot \,:\, \dH^k(M;\bbZ) \times \dH^l(M,S;\bbZ)	 \longrightarrow \dH^{k + l}(M,S;\bbZ)~, \qquad (h,h'\,) \longmapsto h \cdot h'\,~,
\end{equation}
which provides us with the desired module structure.
\sk

Let $h \in \dH^k(M;\bbZ)$ and $h'\, \in \dH^l(M,S;\bbZ)$.
Let $\tilde h \in C^{k-1}(M;\bbR)$ be such that $\tilde h \mod \bbZ = h$ on $Z_{k-1}(M)$ 
and $c_h \in Z^k(M;\bbZ)$ such that $\cdel \tilde h = \int_\cdot\, \cu\, h - c_h \in C^k(M;\bbR)$. 
The pair $(\tilde h,c_h)$ is unique up to a 
term of the form $(\Delta + \cdel \Gamma,- \cdel \Delta)$, 
where $\Delta \in C^{k-1}(M;\bbZ)$ and $\Gamma \in C^{k-2}(M;\bbR)$. 
Similarly, let $\tilde h'\, \in C^{l-1}(M,S;\bbZ)$ be such that $\tilde h'\, \mod \bbZ = h'\,$ 
on $Z_{l-1}(M,S)$ and $c_{h'} \in Z^l(M,S;\bbZ)$ 
such that $\cdel \tilde h'\, = \int_\cdot\, \cu\, h'\, - c_{h'} \in C^l(M,S;\bbR)$. 
The pair $(\tilde h'\,, c_{h'})$  is unique up to a term of 
the form $(\Delta + \cdel \Gamma,-\cdel\Delta)$, 
where $\Delta \in C^{l-1}(M,S;\bbZ)$ and $\Gamma \in C^{l-2}(M,S;\bbR)$. 
\sk

As in Subsection \ref{subsecCS}, we choose a family of natural cochain homotopies
$B: \Omega^k(M) \times \Omega^l(M) \to C^{k+l-1}(M;\bbR)$,
for all $k,l\geq 0$, between the wedge product for differential forms 
and the cup product for the corresponding cochains 
such that $B(\omega,\theta) = (-1)^{k \,l}\, B(\theta,\omega)$, 
for all $\omega \in \Omega^k(M)$ and $\theta \in \Omega^l(M)$. 
These cochain homotopies imply the identities
\begin{equation}\label{eqWedgeCupHomotopy}
\int_\cdot\, \omega \wedge \theta - \int_\cdot\, \omega \smile \int_\cdot\, \theta 
= \cdel B(\omega,\theta) + B(\dd \omega,\theta) + (-1)^k\, B(\omega,\dd \theta)~,
\end{equation}
for all $\omega \in \Omega^k(M)$ and $\theta \in \Omega^l(M)$.
Recall that an example of such a family of cochain homotopies is given 
in \cite[Section 1]{CS}. It is obtained by iterated subdivisions,
the natural chain homotopy between subdivision and identity,
and by exploiting a result due to Kervaire~\cite{Kervaire}. 
With this choice one explicitly observes that $B$ preserves supports, i.e.\ 
$B(\omega,\theta) \in C^{k+l-1}(M,S)$ where $S= M\setminus (\supp\, \omega \cap \supp\,  \theta)$
is the complement of the intersection of the supports of $\omega$ and $\theta$.
More abstractly, this fact follows from naturality of $B$.
As stressed in~\cite[Section 3.2]{HS}, two different choices of $B$ are naturally cochain homotopic. 
This result is crucial in showing that the ring structure on differential characters 
does not depend on the choice of a natural cochain homotopy $B$. Similarly, 
it will allow us to show that the module structure of differential characters on relative cycles
over the ring of differential characters does not depend on the choice of $B$. 
\sk

Recalling that the cup product between cochains preserves supports, 
we define $h \cdot h'\,$ in \eqref{eqRelDiffCharModule} as the homomorphism 
$Z_{k + l - 1}(M,S)\to\bbT$  given by
\begin{subequations}\label{eqRelModuleAction}
\begin{align}\label{eqRelModuleAction1}
h \cdot h'\,		& := \tilde h \smile \int_\cdot \, \cu\, h'\, + (-1)^k\, c_h \smile \tilde h'\, + B(\cu\, h,\cu\, h'\,) \mod \bbZ	\\[4pt]
			& \phantom{:}= \tilde h \smile c_{h'} + (-1)^k\, \int_\cdot\, \cu\, h \smile \tilde h'\, + B(\cu\, h,\cu\, h'\,)
					\mod \bbZ~,
\label{eqRelModuleAction2}
\end{align}
\end{subequations}
where $(\tilde h,c_h)$ and $(\tilde h'\,,c_{h'})$ were introduced above.
The two expressions in \eqref{eqRelModuleAction} differ by the exact term $(-1)^{k+1}\, \cdel(\tilde h \smile \tilde h'\,) \in B^{k+l-1}(M,S;\bbR)$, 
which is of course trivial when evaluated on $Z_{k + l - 1}(M,S)$.
While \eqref{eqRelModuleAction1} shows that $h \cdot h'\,$ does not depend on the choice of $(\tilde h'\,,c_{h'})$, the expression
\eqref{eqRelModuleAction2} shows independence with respect to the choice of $(\tilde h,c_h)$. 
The fact that different choices of $B$ are naturally cochain homotopic entails that 
$h \cdot h'\,$ does not depend on this choice as well.
It is easy to check that $h \cdot h'\,$ as defined above is an element of $\dH^{k+l}(M,S;\bbZ)$: using \eqref{eqWedgeCupHomotopy} one finds that
\begin{equation}
(h \cdot h'\,)(\del \gamma) = \int_\gamma \, \cu\, h \wedge \cu\, h'\, \mod \bbZ~,
\end{equation}
for all $\gamma \in C_{k+l}(M)$.
Because the formula for the module structure on relative differential cohomology 
is exactly the same as the one for the ring structure on differential cohomology, 
one can easily adapt the arguments in \cite[Theorem 1.11]{CS} to the present case
and show that  \eqref{eqRelDiffCharModule} structures
$\dH^\sharp(M,S;\bbZ)$ into a module over $\dH^\sharp(M;\bbZ)$. 
For an alternative approach see \cite[Part II, Section 4.2.6]{BB}. 
Directly from \eqref{eqRelModuleAction}, one can prove the identities
\begin{subequations}\label{eqRelCompatibility}
\begin{align}
\cu(h \cdot h'\,)			&	= \cu\, h \wedge \cu\, h'\,~,
								&	\ch(h \cdot h'\,)			&	= \ch\, h \smile \ch\, h'\,~,		\\[4pt]
h \cdot \kappa\, \upsilon	&	= (-1)^k \,\kappa(\ch\, h \smile \upsilon)~,
								&	h \cdot \iota\, [\alpha]	&	= (-1)^k \, \iota\, [\cu\, h \wedge \alpha]~,	\\[4pt]
\kappa\, u \cdot h'\,		&	= \kappa(u \smile \ch\, h'\,)~,
								&	\iota\, [A] \cdot h'\,		&	= \iota\, [A \wedge \cu\, h'\,]~,
\end{align}
\end{subequations}
for all $h {\in} \dH^k(M;\bbZ)$, $h'\, {\in} \dH^l(M,S;\bbZ)$, $[A] \in \Omega^{k-1}(M) / \Omega^{k-1}_\bbZ(M)$,
$[\alpha] \in \Omega^{l-1}(M,S) / \Omega^{l-1}_\bbZ(M,S)$, $u \in \H^{k-1}(M;\bbT)$ and
$\upsilon \in \H^{l-1}(M,S;\bbT)$,
which express the compatibility of the module structure for differential characters on relative cycles
with respect to the natural homomorphisms $\cu$, $\ch$, $\iota$ and $\kappa$.
\sk

We conclude by noticing that the $\dH^\sharp(M;\bbZ)$-module structure on
$\dH^\sharp(M,S;\bbZ)$ is natural with respect to morphisms 
$f:(M,S) \to (M^\prime,S^\prime\, )$ in the category $\Pair$, i.e.\
the pull-back $f^\ast: \dH^k(M^\prime,S^\prime;\bbZ) \to \dH^k(M,S;\bbZ)$ is a module homomorphism 
with underlying ring homomorphism $f^\ast: \dH^k(M^\prime;\bbZ) \to \dH^k(M;\bbZ)$. 
For this, let $h \in \dH^k(M^\prime;\bbZ)$ and $h'\, \in \dH^l(M^\prime,S^\prime;\bbZ)$ 
and choose cochains $\tilde h \in C^{k-1}(M^\prime;\bbR)$ and $\tilde h'\, \in C^{l-1}(M^\prime,S^\prime;\bbR)$ 
which extend and lift $h$ and $h'\,$, respectively. Take
$c_h \in Z^k(M^\prime;\bbZ)$ and $c_{h'} \in Z^l(M^\prime,S^\prime;\bbZ)$ 
such that $\cdel \tilde h = \int_\cdot\, \cu\, h - c_h$ and $\cdel \tilde h'\, = \int_\cdot\, \cu\, h'\, - c_{h'}$. 
Then $f^\ast \tilde h \in C^{k-1}(M;\bbR)$ and $f^\ast \tilde h'\, \in C^{l-1}(M,S;\bbR)$ 
extend and lift $f^\ast h\in \dH^k(M;\bbZ)$ and $f^\ast h'\, \in \dH^l(M,S;\bbZ)$, respectively. 
Furthermore, $f^\ast c_h \in Z^k(M;\bbZ)$ and $f^\ast c_{h'} \in Z^l(M,S;\bbZ)$ 
satisfy $\cdel f^\ast \tilde h = \int_\cdot\, \cu\, f^\ast h - f^\ast c_h$ 
and $\cdel f^\ast \tilde h'\, = \int_\cdot\, \cu\, f^\ast h'\, - f^\ast c_{h'}$. 
Computing $f^\ast h \cdot f^\ast h'\,$ using \eqref{eqRelModuleAction},
we conclude that $f^\ast h \cdot f^\ast h'\, = f^\ast(h \cdot h'\,)$
because the cup product $\smile$ and the cochain homotopy $B$ are natural.

\begin{rem}\label{remRelDiffCharMultiplication}
We observe that \eqref{eqRelModuleAction} also makes sense 
when both factors are differential characters on relative cycles: For
$S$, $S^\prime$ and $S \cup S^\prime$ submanifolds (possibly with boundary) of $M$, 
the same arguments would show that 
\begin{equation}\label{eqRelDiffCharProd}
\cdot \, :\, \dH^k(M,S;\bbZ) \times \dH^l(M,S^\prime;\bbZ)
\longrightarrow \dH^{k+l}(M,S \cup S^\prime;\bbZ)~, \qquad 
(h,h'\,) \longmapsto h \cdot h'\,
\end{equation}
is a well-defined bihomomorphism.
For $S = S^\prime$, \eqref{eqRelDiffCharProd} defines
a graded-commutative ring structure (without unit if $S$ is non-empty)
on differential characters on relative cycles. 
This ring structure coincides with the usual ring structure on differential characters 
for $S = S^\prime = \emptyset$. Furthermore, when $S \subseteq S^\prime$, 
we can interpret $\dH^l(M,S^\prime;\bbZ)$ as a module 
over the ring $\dH^k(M,S;\bbZ)$ (without unit for $S \neq \emptyset$). 
When $S = \emptyset$, this coincides
with the module structure introduced in \eqref{eqRelDiffCharModule}. 
\end{rem}


\section{Differential characters with compact support}\label{secCDiffChar}
To introduce differential characters with compact support, 
we follow an approach similar to the one used in Subsection \ref{secHc}
to define ordinary cohomology with compact support.
Let $M$ be a manifold. Consider its associated directed set $\mcK_M$ of compact subsets $K\subseteq M$
and introduce the functor $(M,M \setminus -): \mcK_M \to \Pair^\op$ as in \eqref{eqCompSubsetToPairFunctor}.
Composing this functor with the relative differential cohomology functor 
$\dH^k(-;\bbZ): \Pair^\op \to \Ab$ results in the functor
\begin{equation}\label{eqn:dHdiagram}
\dH^k(M,M \setminus -;\bbZ): \mcK_M \longrightarrow \Ab~.
\end{equation}
\begin{defi}\label{defCDiffChar}
The Abelian group of differential characters with compact support 
is the colimit
\begin{equation}\label{eqCDiffChar}
\dH^k_{\rm c}(M;\bbZ) := \colim\big(\dH^k(M,M \setminus -;\bbZ): \mcK_M \to \Ab\big)~
\end{equation}
of the functor \eqref{eqn:dHdiagram} over the directed set $\mcK_M $.
\end{defi}

As for ordinary cohomology (cf.\ Remark \ref{remAlternativeCoho}),
the Abelian groups $\dH^k_{\rm c}(M;\bbZ)$ can also be computed
as a colimit over the directed set $\mcO^{\rm c}_M$ instead of
$\mcK_M$ (since both $\mcO^{\rm c}_M$ and $\mcK_M$ are cofinal in a larger directed set $\mcU_{M}$). 
Composing the functor $(M,M \setminus -): \mcO^{\rm c}_M \to \PePair^\op$
with the embedding $\PePair^\op\to \Pair^\op$ and 
$\dH^k(-;\bbZ): \Pair^\op \to \Ab$, we obtain another functor 
$\dH^k(M,M \setminus -;\bbZ): \mcO^{\rm c}_M \to \Ab$
whose colimit is isomorphic to $\dH^k_{\rm c}(M;\bbZ)$, i.e.\
\begin{equation}\label{eqn:alternativedHc}
\dH^k_{\rm c}(M;\bbZ)\simeq \colim\big(\dH^k(M,M \setminus -;\bbZ): \mcO^{\rm c}_M \to \Ab\big)~.
\end{equation}
A technical advantage of this alternative point of view
is that $\dH^k(M,M \setminus O;\bbZ)$, for any $O\in \mcO^{\rm c}_M$,
fits into the full commutative diagram 
of short exact sequences \eqref{eqRelDiffCharDiaAlt} 
while $\dH^k(M,M \setminus K;\bbZ)$, for $K\in \mcK_M$,
in general just fits into the incomplete diagram \eqref{eqRelDiffCharDia}.
\sk

We next define the Abelian group
\begin{equation}\label{eqColimFormsZ}
\Omega^k_{{\rm c},\bbZ}(M) := \colim \big(\Omega^k_\bbZ(M,M \setminus -): \mcK_M \to \Ab\big)
\end{equation}
of (closed) $k$-forms with compact support having integral periods on cycles relative to the complement of their support.
Recalling that $\colim$ is an exact functor for diagrams in Abelian groups over directed sets, it follows that 
\begin{equation}
\colim \Big(\, \frac{\Omega^k(M,M \setminus -)}{\OmegaZ^k(M,M
  \setminus -)}: \mcK_M \to \Ab \, \Big) 
= \frac{\Omega^k_{\rm c}(M)}{\Omega^k_{{\rm c},\bbZ}(M)}~.
\end{equation}
\begin{theo}\label{thmCDiffChar}
Let $M$ be a manifold. Then the diagram
\begin{equation}\label{eqCDiffCharDia}
\xymatrix{
		&	0 \ar[d]	&	0 \ar[d]	&	0 \ar[d]																\\
0 \ar[r]	&	\frac{\H^{k-1}_{\rm c}(M;\bbR)}{\H_{{\rm c},\free}^{k-1}(M;\bbZ)} \ar[r] \ar[d]
					&	\frac{\Omega^{k-1}_{\rm c}(M)}{\Omega^{k-1}_{{\rm c},\bbZ}(M)} \ar[r]^-\dd \ar[d]_-\iota
								&	\dd \Omega^{k-1}_{\rm c}(M) \ar[r] \ar[d]					&	0			\\
0 \ar[r]	&	\H^{k-1}_{\rm c}(M;\bbT) \ar[r]^-\kappa \ar[d]	
					&	\dH^k_{\rm c}(M;\bbZ) \ar[r]^-\cu \ar[d]_-\ch
								&	\Omega^k_{{\rm c},\bbZ}(M) \ar[r] \ar[d]					&	0			\\
0 \ar[r]	&	\H^k_{{\rm c},\tor}(M;\bbZ) \ar[r] \ar[d]
					&	\H^k_{\rm c}(M;\bbZ) \ar[r] \ar[d]
								&	\H^k_{{\rm c},\free}(M;\bbZ) \ar[r] \ar[d]					&	0			\\
		&	0		&	0		&	0
}
\end{equation}
commutes and its rows and columns are short exact sequences.
\end{theo}
\begin{proof}
This result follows immediately from \eqref{eqn:alternativedHc} 
and Theorem~\ref{thmRelDiffCharRefined} because
the complement $M \setminus O$ of any $O \in \mcO^{\rm c}_M$ is properly embedded
and $\colim$ is an exact functor for diagrams in Abelian groups over directed sets.
\end{proof}

\paragraph*{Functoriality:}
We show that
\begin{equation}\label{eqFunctoriality}
\dH^k_{\rm c}(-;\bbZ): \Man_{m,\emb} \longrightarrow \Ab
\end{equation}
is a functor. Given an open embedding $f: M \to M^\prime$ between $m$-dimensional manifolds,
one can use the same arguments as in Subsection \ref{secHc} and the Excision Theorem \ref{thmDiffCharExcision}
for differential characters on relative cycles to conclude that
$f^\ast: \dH^k(M^\prime,M^\prime \setminus -;\bbZ) \circ \mcK_f \Rightarrow \dH^k(M,M \setminus -;\bbZ)$ 
is a natural isomorphism. Denoting its inverse by
\begin{equation}
(f^\ast)^{-1}: \dH^k(M,M \setminus -;\bbZ) \Longrightarrow \dH^k(M^\prime,M^\prime \setminus -;\bbZ) \circ \mcK_f~,
\end{equation}
the universal property of the colimit allows us to define a 
canonical homomorphism $f_\ast: \dH^k_{\rm c}(M;\bbZ) \to \dH^k_{\rm c}(M^\prime;\bbZ)$,
which we call the push-forward of differential characters with compact support.
One can then easily check that $\dH^k_{\rm c}(-;\bbZ): \Man_{m,\emb} \to \Ab$ is a functor
and that naturality of $\cu$, $\ch$, $\iota$ and $\kappa$ for differential characters on relative chains
carries over to the compactly supported case via our colimit prescription, i.e.\ 
\begin{subequations}
\begin{align}
\cu: \dH^k_{\rm c}(-;\bbZ) \Longrightarrow \Omega^k_{{\rm c},\bbZ}(-)~,		&&	\ch: \dH^k_{\rm c}(-;\bbZ) \Longrightarrow  \H^k_{\rm c}(-;\bbZ)~,		\\[4pt]
\iota: \frac{\Omega^{k-1}_{\rm c}(-)}{\Omega^{k-1}_{{\rm c},\bbZ}(-)} \Longrightarrow  \dH^k_{\rm c}(-;\bbZ)~,
													&&	\kappa: \H_{\rm c}^{k-1}(-;\bbT) \Longrightarrow  \dH^k_{\rm c}(-;\bbZ)~,
\end{align}
\end{subequations}
are natural transformations between functors from $\Man_{m,\emb}$ to $\Ab$.

\begin{ex}
As explained in Examples~\ref{ex:linebundles} and \ref{ex:parallellinebundles}, 
(relative) differential cohomology groups in degree $k=2$ are canonically isomorphic
to isomorphism classes of  Hermitean line bundles with connection (and parallel section).
By passing to the colimit over the directed set $\mcO^{\rm c}_M$ we obtain a 
canonical identification of the compactly supported differential cohomology group $\dH^2_{\rm c}(M;\bbZ)$ 
with the group of isomorphism classes of  Hermitean line bundles with connection 
and a parallel section outside some relatively compact open subset $O \subseteq M$.
Two such triples with sections $\sigma:M \setminus O \to L$ and $\sigma^\prime : M \setminus O^\prime \to L^\prime$ 
are isomorphic if and only if there exists a bundle isomorphism $\psi:L \to L^\prime$,
a subset $\tilde O \in \mcO^{\rm c}_M$ with $\tilde O \subset O \cap O^\prime$, and a section $\tilde\sigma : M \setminus \tilde O \to L$ 
such that 
$\tilde \sigma|_{M\setminus O} = \sigma$ and 
$\psi \circ \tilde\sigma|_{M \setminus O^\prime} = \sigma^\prime$.
\end{ex}

\paragraph*{Module structure:}
Using the natural module structure for differential characters on relative cycles
developed in Subsection \ref{subRelDiffCharModule}, our colimit
prescription for differential characters with compact support
given in Definition \ref{defCDiffChar} yields a natural module 
structure for $\dH^\sharp_{\rm c}(M;\bbZ)$ over the ring $\dH^\sharp(M;\bbZ)$.
We denote this module structure by
\begin{equation}\label{eqn:modcomp}
\cdot\,:\, \dH^k(M;\bbZ) \times \dH^l_{\rm c}(M;\bbZ)
\longrightarrow \dH^{k + l}_{\rm c}(M;\bbZ)~, \qquad (h,h'\,) \longmapsto h \cdot h'\,~.
\end{equation}
The bihomomorphism \eqref{eqn:modcomp} is obtained by taking the 
colimit over $K \in \mcK_M$ in \eqref{eqRelDiffCharModule} for $S = M \setminus K$. 
Given any morphism $f:M \to M^\prime$ in $\Man_{m,\emb}$, 
the diagram 
\begin{equation}\label{eqCDiffCharModuleNat}
\xymatrix@C=40pt{
\dH^k(M^\prime;\bbZ) \times \dH^l_{\rm c}(M;\bbZ) \ar[r]^-{f^\ast \times \id} \ar[d]_-{\id \times f_\ast}
		&	\dH^k(M;\bbZ) \times \dH^l_{\rm c}(M;\bbZ) \ar[r]^-{\cdot}
				&	\dH^{k+l}_{\rm c}(M;\bbZ) \ar[d]^-{f_\ast}												\\
\dH^k(M^\prime;\bbZ) \times \dH^l_{\rm c}(M^\prime;\bbZ) \ar[rr]_-{\cdot}
		&		&	\dH^{k+l}_{\rm c}(M^\prime;\bbZ)
}
\end{equation}
commutes by construction of $f_\ast$, which implies that the $\dH^\sharp(M;\bbZ)$-module
structure on $\dH^\sharp_{\rm c}(M;\bbZ)$ is natural.
The homomorphisms $\cu$, $\ch$, $\iota$ and $\kappa$ 
for compactly supported differential characters are compatible with the module structure, i.e.\
\begin{subequations}\label{eqCCompatibility}
\begin{align}
\cu(h \cdot h'\,)			&	= \cu\, h \wedge \cu\, h'\,~,
								&	\ch(h \cdot h'\,)			&	= \ch\, h \smile \ch\, h'\,~,					\\[4pt]
h \cdot \kappa\, \upsilon	&	= (-1)^k\,  \kappa(\ch\, h \smile \upsilon)~,
								&	h \cdot \iota\, [\alpha]	&	= (-1)^k\,  \iota\, [\cu\, h \wedge \alpha]~,	\\[4pt]
\kappa\, u \cdot h'\,		&	= \kappa(u \smile \ch\, h'\,)~,
								&	\iota\, [A] \cdot h'\,		&	= \iota\, [A \wedge \cu\, h'\,]~,
\end{align}
\end{subequations}
for all  $h \in \dH^k(M;\bbZ)$, $h'\, \in \dH^l_{\rm c}(M;\bbZ)$, 
$[A] \in \Omega^{k-1}(M) / \Omega^{k-1}_\bbZ(M)$, $[\alpha] \in \Omega^{l-1}_{\rm c}(M) / \Omega^{l-1}_{{\rm c},\bbZ}(M)$,
 $u \in \H^{k-1}(M;\bbT)$ and $\upsilon \in \H^{l-1}_{\rm c}(M;\bbT)$.
This again follows from similar results for the module structure of differential characters on relative cycles, 
see Subsection \ref{subRelDiffCharModule}.

\paragraph*{The natural homomorphism $\boldsymbol I$:}
Applying our colimit prescription to the homomorphism displayed in \eqref{eqI}
defines a natural homomorphism
\begin{equation}\label{eqCDiffCharToDiffChar}
I: \dH^k_{\rm c}(M;\bbZ) \longrightarrow \dH^k(M;\bbZ)
\end{equation}
sending differential characters with compact support to Cheeger-Simons differential characters. 
Naturality means that for any morphism $f : M\to M^\prime$
in $\Man_{m,\emb}$ the diagram
\begin{equation}\label{eqCDiffCharToDiffCharNat}
\xymatrix{
\dH^k_{\rm c}(M;\bbZ) \ar[r]^-I \ar[d]_-{f_\ast}	&	\dH^k(M;\bbZ)							\\
\dH^k_{\rm c}(M^\prime;\bbZ) \ar[r]_-I			&	\dH^k(M^\prime;\bbZ) \ar[u]_-{f^\ast}
}
\end{equation}
commutes. As for differential characters on relative cycles, this homomorphism 
is compatible with curvature, characteristic class, topological trivialization 
and inclusion of flat characters. It is furthermore multiplicative with respect to the 
$\dH^\sharp(M;\bbZ)$-module structure of $\dH^\sharp_{\rm c}(M;\bbZ)$.
\begin{rem}\label{remNonInj}
As for ordinary cohomology, the natural homomorphism 
$I: \dH^k_{\rm c}(M;\bbZ) \to \dH^k(M;\bbZ)$ is in general neither surjective nor injective.
Its kernel and cokernel may be characterized by applying the colimit over the directed 
set $\mcO^{\rm c}_M$ to the long exact sequence \eqref{eq:long_ex_sequ_par}.
We observe that injectivity of $I: \dH^k_{\rm c}(M;\bbZ) \to \dH^k(M;\bbZ)$ 
is equivalent to injectivity of $\H^{k-1}_{\rm c}(M;\bbT) \to \H^{k-1}(M;\bbT)$. 
This follows from the colimit of the diagram in Remark \ref{remRelNonSub} 
and observing that $\Omega^k_{{\rm c},\bbZ}(M) \to \OmegaZ^k(M)$ is an injection. 
We provide below some examples for which $\H^{k-1}_{\rm c}(M;\bbT) \to \H^{k-1}(M;\bbT)$ 
fails to be injective, and therefore $I: \dH^k_{\rm c}(M;\bbZ) \to \dH^k(M;\bbZ)$ too: 
Let $m \geq 3$, $k \in \{2, \ldots, m\}$ with $2k \neq m+2, m+3$ 
and take $M = \bbR^{k-1} \times \bbS^{m-k+1}$. 
Observe that the collection of compact subsets $\mathcal B = \{B_r \times \bbS^{m-k+1}: r>0\}$, 
where $B_r$ is the closed ball in $\bbR^{k-1}$ of radius $r$ centered at the origin, is cofinal in $\mcK_M$. 
Therefore, recalling \eqref{eqHcdef}, we find 
$\H^{k-1}_{\rm c}(M;\bbT) \simeq \colim(\H^{k-1}(M, M \setminus -;\bbT): \mathcal B \to \Ab)$. 
The long exact sequence relating relative and absolute cohomology groups, 
see e.g.\ \cite[p.\ 200]{Hatcher}, provides the exact sequence 
\begin{equation}
\xymatrix{
\H^{k-2}(M;\bbT) \ar[r]	&	\H^{k-2}(M \setminus K;\bbT) \ar[r]
&	\H^{k-1}(M, M \setminus K;\bbT) \ar[r]	&	\H^{k-1}(M;\bbT)
}
\end{equation}
for each $K \in \mathcal B$. Since $2k \neq m+2$, $\H^{k-1}(M;\bbT)$ is trivial; 
therefore $\H^{k-1}(M, M \setminus K;\bbT)$ is isomorphic to the quotient of $\H^{k-2}(M \setminus K;\bbT)$ 
by the image of $\H^{k-2}(M;\bbT)$. 
Taking also $2k \neq m+3$ into account, one concludes that $\H^{k-2}(M \setminus K;\bbT) \simeq \bbT$. 
Since this result is the same for all $K \in \mathcal B$, the homomorphism 
$\bbT \simeq \H^{k-1}_{\rm c}(M;\bbT) \to \H^{k-1}(M;\bbT) \simeq 0$ is not injective. 
\end{rem}
\begin{rem}\label{remHLZ}
Since differential cohomology groups are usually regarded as a refinement
of smooth singular cohomology groups by differential forms, it seems reasonable 
to introduce compactly supported differential cohomology groups as a refinement of 
compactly supported cohomology groups by compactly supported differential forms.
This is exactly the case for our notion of compactly supported differential 
characters $\dH^\sharp_{\rm c}(M;\bbZ)$, as well as for the model 
in \cite[Section 8]{HLZ}, where the authors develop 
another approach to differential cohomology with compact support
based on de Rham-Federer currents. 
In particular, \cite[Proposition 8.3]{HLZ} is the analogue of our Theorem \ref{thmCDiffChar}. 
However, in \cite[Theorem 8.4]{HLZ} it is erroneously stated that the group of de 
Rham-Federer characters with compact support 
is isomorphic to the \emph{subgroup} of Cheeger-Simons differential characters on $M$ 
vanishing upon restriction to the complement of some compact subset $K \subseteq M$. 
This statement is in contrast with our results, see also Remark~\ref{remNonSub} below. 
Actually, assuming that an Abelian group $\dH^k_{\rm c}(M;\bbZ)$ fits into the commutative diagram
\begin{equation}
\xymatrix{
0 \ar[r] & \H^{k-1}_{\rm c}(M;\bbT) \ar[r] \ar[d] & \dH^k_{\rm c}(M;\bbZ) \ar[r] \ar[d]_-{I} & \Omega^k_{{\rm c},\bbZ}(M) \ar[r] \ar[d] & 0 \\
0 \ar[r] & \H^{k-1}(M;\bbT) \ar[r] & \dH^k(M;\bbZ) \ar[r] & \OmegaZ^k(M) \ar[r] & 0 
}
\end{equation}
with exact rows, the non-injectivity of the homomorphism $\H^{k-1}_{\rm c}(M;\bbT) \to \H^{k-1}(M;\bbT)$ implies that $\dH^k_{\rm c}(M;\bbZ)$ cannot be a subgroup of $\dH^k(M;\bbZ)$ in general.
On the other hand, from the long exact sequence \eqref{eq:long_ex_sequ_par} we conclude that the inclusion of the subgroup of Cheeger-Simons characters on $M$ which vanish upon restriction to the complement of some relatively compact open subset $O \in \mcO^{\rm c}_M$ into $\dH^k(M;\bbZ)$ factorizes the map $I:\dH^k_{\rm c}(M;\bbZ) \to \dH^k(M;\bbZ)$.
In other words, the subgroup of $\dH^k(M;\bbZ)$ of Cheeger-Simons differential characters introduced in \cite[Theorem 8.4]{HLZ} coincides with the image of the homomorphism $I:\dH^k_{\rm c}(M;\bbZ) \to \dH^k(M;\bbZ)$. 
It follows that the homomorphism of \cite[Theorem 8.4]{HLZ} in general fails to be an isomorphism. 
Let us stress that this fact does not affect the rest of \cite{HLZ}. 
\end{rem}

\begin{rem}
For the sake of completeness, it is worth to compare 
between the compactly supported differential characters 
introduced here using the Cheeger-Simons approach to differential cohomology 
and the de Rham-Federer characters with compact support established in \cite[Section 8]{HLZ}.\footnote{We 
are grateful to the anonymous referee for encouraging us to pursue this point 
and for suggesting the relevant construction.} 
Here we only sketch the construction of a natural comparison isomorphism, 
pointing to the relevant literature. A more detailed argument requires the introduction 
of a large amount of material from de Rham-Federer theory, 
which is beyond the scope of the present paper. 
The relevant homomorphism can be obtained in analogy with \cite[Section 4]{HLZ}. 
Specifically, integrating a compactly supported de Rham-Federer character 
over cycles relative to its support leads to a compactly supported differential character 
in the sense of the present paper. This provides a homomorphism 
\begin{equation}
\Psi: \hat{\mathbb{H}}_{\mathrm{cpt}}^{k-1}(M) \longrightarrow \dH_{\rm c}^k(M;\bbZ)~,
\end{equation}
where $\hat{\mathbb{H}}_{\mathrm{cpt}}^p(M)$ denotes 
the Abelian group of de Rham-Federer character with compact support, cf.\ \cite[Definition 8.2]{HLZ}. 
(Notice the degree shift due to the different convention adopted in \cite{HLZ}.) 
$\Psi$ behaves naturally with respect to open embeddings, cf.\ \eqref{eqFunctoriality}. 
To prove that $\Psi$ is also an isomorphism, one considers the short exact sequence 
in \cite[Proposition 8.3]{HLZ} and the similar one obtained from \eqref{eqCDiffCharDia}. 
Poincar\'e duality between homology and cohomology with compact support \cite[Theorem 3.35]{Hatcher} 
provides isomorphisms between the objects appearing on the left (respectively on the right) 
of the sequences mentioned above. One can check that 
these isomorphisms, together with $\Psi$, form a morphism of short exact sequences. 
Therefore, by the five lemma, the natural homomorphism $\Psi$ is also an isomorphism. 
\end{rem}

\begin{ex}
A source of examples of compactly supported differential
characters is obtained by classical Chern-Weil theory.
As in Example~\ref{ex:CW}, let $G$ be a Lie group with finitely many 
connected components and $\lambda: \mathfrak g^k \to \bbR$ an invariant polynomial.
Let $P \to M$ be a principal $G$-bundle with connection $\theta$
and suppose that there exists an isomorphism $(P,\theta)|_{M \setminus O} \xrightarrow{\simeq} ((M\setminus O) \times G,\dd)$ 
to the trivial bundle with trivial connection outside some relatively compact open 
subset $O \in \mcO^{\rm c}_M$.\footnote{Since $\theta$ is equivalent to the trivial connection outside $O$, 
the constant maps in the trivialization are parallel sections.}
Mimicking in this setting the construction from \cite{CS} 
of differential characters with curvature the Chern-Weil form $\lambda(F_\theta)$, 
we obtain an element of $\dH^{2k}(M;\bbZ)$ that lies in the image of 
the natural homomorphism $I: \dH^{2k}_{\rm c}(M;\bbZ) \to \dH^{2k}(M;\bbZ)$.
Since $I$ is in general not injective, 
the question arises whether this character has a canonical preimage.
Taking into account the identification of $\dH^{2k}(M,M\setminus O;\bbZ)$ 
with the group of parallel relative differential characters from \cite{BB}, 
we obtain a canonical character from the \emph{Cheeger-Chern-Simons character} 
associated with $(P,\theta)$,  see \cite{B14} for details of the construction. 
\end{ex}


\section{Smooth Pontryagin duality}\label{secPontrDuality}
The aim of this section is to establish a version of Pontryagin duality for Cheeger-Simons differential characters. 
This should be compared to \cite{HLZ}, where a similar result is obtained 
in a different model for differential cohomology based on de Rham-Federer currents.
\sk
 
In this section all manifolds are implicitly assumed to be connected, $m$-dimensional, and oriented.
This in particular allows us to define the integration homomorphism $\int_M : \Omega_{\rm c}^m(M)\to \bbR$.
We denote by $\oMan_{m,\emb}$ the corresponding category of
oriented and connected $m$-dimensional manifolds with morphisms given by orientation-preserving 
open embeddings. Some of the following results are proven under the additional 
(sufficient but not necessary) hypothesis that the manifold $M$ is of finite-type, which 
implies that all (co)homology groups are finitely generated and in particular allows us to 
interpret cohomology with compact support as the dual of ordinary cohomology, 
see e.g.\ Poincar\'e duality for de Rham cohomology in \cite[Chapter~I, Section~5]{BT}.
We will clearly indicate which statements rely on this
assumption. Some of the technical details required in
Section~\ref{subPontryagin} are delegated to
Appendix~\ref{app:axuiliary} at the end of the paper.

\subsection{Definitions}\label{subPontryagin}
For any Abelian group $G$, we denote its Pontryagin dual
(also called the character group) by $G^\star:=\Hom(G,\bbT)$. 
We define the smooth Pontryagin dual $\Omega^k_{\rm c}(M)^\star_\infty\subseteq \Omega^k_{\rm c}(M)^\star$ 
of $\Omega^k_{\rm c}(M)$ as the subgroup of elements $\varphi\in \Omega^k_{\rm c}(M)^\star$ 
which are smooth in the sense that there exists $\omega \in \Omega^{m-k}(M)$
such that
\begin{equation}
\varphi(\alpha) = \int_M \, \omega \wedge \alpha \mod \bbZ~,
\end{equation}
for all $\alpha \in \Omega_{\rm c}^k(M)$. Similarly, we define the smooth Pontryagin dual $\Omega^k(M)^\star_\infty\subseteq \Omega^k(M)^\star$ 
of $\Omega^k(M)$ as the subgroup of elements $\psi\in \Omega^k(M)^\star$ 
which are smooth in the sense that there exists $\alpha \in \Omega^{m-k}_{\rm c}(M)$
such that 
\begin{equation}
\psi(\omega) = \int_M \, \omega \wedge \alpha \mod \bbZ~,
\end{equation}
for all $\omega \in \Omega^k(M)$. Introducing the (weakly non-degenerate) $\bbT$-valued pairing
\begin{equation}\label{eqPairingFormPre}
\ips{\cdot}{\cdot}_{\Omega}^{} : \Omega^k(M)\times \Omega^{m-k}_{\rm c}(M)\longrightarrow \bbT~, \qquad 
(\omega,\alpha) \longmapsto (-1)^k\,\int_M\, \omega\wedge\alpha \mod \bbZ~,
\end{equation}
partial evaluation provides us with the two homomorphisms
\begin{subequations}\label{eqn:partialevalpre}
\begin{flalign}
\Omega^k(M) \longrightarrow \Omega^{m-k}_{\rm c}(M)^\star~,& \qquad \omega \longmapsto \ips{\omega}{\cdot}_{\Omega}^{}~,\\[4pt]
\Omega^{m-k}_{\rm c}(M) \longrightarrow \Omega^k(M)^\star ~,& \qquad \alpha \longmapsto \ips{\cdot}{\alpha}_{\Omega}^{}~.
\end{flalign}
\end{subequations}
It is clear from these definitions that \eqref{eqn:partialevalpre} induces isomorphisms
\begin{equation}
\Omega^k(M) \simeq \Omega^{m-k}_{\rm c}(M)^\star_\infty~, \qquad \Omega^{m-k}_{\rm c}(M)\simeq \Omega^k(M)^\star_\infty~.
\end{equation}

We further define the smooth Pontryagin duals of the quotients
$\Omega^k_{\rm c}(M) / \Omega^k_{{\rm c},\bbZ}(M)$
and $ \Omega^k(M) / \OmegaZ^k(M)$ by
\begin{subequations}\label{eqSmoothPontrDualTT}
\begin{align}
\Big(\, \frac{\Omega^k_{\rm c}(M)}{\Omega^k_{{\rm c},\bbZ}(M)}\, \Big)^\star_\infty
&	:= \big\{\varphi \in \Omega^k_{\rm c}(M)^\star_\infty\,:\;
		\varphi\big(\Omega^k_{{\rm c},\bbZ}(M)\big) = \{0\} \big\}\,,\\[4pt]
\Big(\, \frac{\Omega^k(M)}{\OmegaZ^k(M)}\, \Big)^\star_\infty
&	:= \big\{\psi \in \Omega^k(M)^\star_\infty\,:\; \psi\big(\OmegaZ^k(M)\big) = \{0\} \big\}\,.
\end{align}
\end{subequations}
Notice that the smooth Pontryagin dual
$(\Omega^k_{\rm c}(M) / \Omega^k_{{\rm c},\bbZ}(M))^\star_{\infty}$ 
can be identified with a subgroup of  $(\Omega^k_{\rm c}(M) / \Omega^k_{{\rm c},\bbZ}(M))^\star$ 
and similarly that $(\Omega^k(M) / \OmegaZ^k(M))^\star_{\infty}$ can 
be identified with a subgroup of $(\Omega^k(M) / \OmegaZ^k(M))^\star$.
We also define the smooth Pontryagin duals of the subgroups
$\Omega^k_{{\rm c},\bbZ}(M)\subseteq \Omega^k_{\rm c}(M)$ and $\OmegaZ^k(M)\subseteq \Omega^k(M)$ by
\begin{subequations}\label{eqSmoothPontrDualFormsZ}
\begin{flalign}
\Omega^k_{{\rm c},\bbZ}(M)^\star_\infty	&	:= \frac{\Omega^k_{\rm c}(M)^\star_\infty}{\big\{\varphi \in \Omega^k_{\rm c}(M)^\star_\infty\,:\;
		\varphi\big(\Omega^k_{{\rm c},\bbZ}(M)\big) = \{0\} \big\}} = \frac{\Omega^k_{\rm c}(M)^\star_\infty}{\Big(\, \frac{\Omega^k_{\rm c}(M)}{\Omega^k_{{\rm c},\bbZ}(M)}\, \Big)^\star_\infty}~,\\[4pt]
\OmegaZ^k(M)^\star_\infty			&	:= \frac{\Omega^k(M)^\star_\infty}{\big\{\psi \in \Omega^k(M)^\star_\infty\,:\; \psi\big(\OmegaZ^k(M)\big) = \{0\} \big\}} = \frac{\Omega^k(M)^\star_\infty}{\Big(\, \frac{\Omega^k(M)}{\OmegaZ^k(M)}\, \Big)^\star_\infty}~.
\end{flalign}
\end{subequations}
Notice that the smooth Pontryagin dual
$\Omega^k_{{\rm c},\bbZ}(M)^\star_\infty$ 
can be identified with a subgroup of  $\Omega^k_{{\rm c},\bbZ}(M)^\star$ 
and similarly that $\OmegaZ^k(M)^\star_\infty$ can 
be identified with a subgroup of $\OmegaZ^k(M)^\star$.
\sk

To characterize \eqref{eqSmoothPontrDualTT} and \eqref{eqSmoothPontrDualFormsZ} more explicitly,
we notice that Lemma \ref{lemFormsZ} implies that
\begin{equation}\label{eqFormsZ}
\OmegaZ^k(M) 
= \Big\{\omega \in \Omega^k(M)\,:\; \int_M\, \omega \wedge \Omega^{m-k}_{{\rm c},\bbZ}(M) \subseteq \bbZ\Big\}~. 
\end{equation}
Hence \eqref{eqPairingFormPre} induces the $\bbT$-valued pairing
\begin{equation}\label{eqPairingFormRightQuotient} 
\ips{\cdot}{\cdot}_{\Omega}^{} : \OmegaZ^k(M)\times \frac{\Omega^{m-k}_{\rm c}(M)}{\Omega^{m-k}_{{\rm c},\bbZ}(M)}
\longrightarrow \bbT~, \qquad (\omega,[\alpha]) \longmapsto (-1)^k\,\int_M\, \omega\wedge\alpha \mod \bbZ
\end{equation}
and by partial evaluation the two homomorphisms
\begin{subequations}\label{eqn:partevalOmega1} 
\begin{flalign}
\label{eqn:partialevalOmegaZ} \OmegaZ^k(M) \longrightarrow \Big(\, \frac{\Omega^{m-k}_{\rm c}(M)}{\Omega^{m-k}_{{\rm c},\bbZ}(M)}\, \Big)^\star~,& \qquad \omega \longmapsto \ips{\omega}{\cdot}_{\Omega}^{}~,\\[4pt]
\label{eqn:partialevalOmegacmodOmegacZ} \frac{\Omega^{m-k}_{\rm c}(M)}{\Omega^{m-k}_{{\rm c},\bbZ}(M)}  \longrightarrow \OmegaZ^k(M)^\star 
~,& \qquad [\alpha] \longmapsto \ips{\cdot}{[\alpha]}_{\Omega}^{}~.
\end{flalign}
\end{subequations}
Moreover, again because of \eqref{eqFormsZ}, the pairing \eqref{eqPairingFormPre} induces another $\bbT$-valued pairing
\begin{equation}\label{eqPairingFormLeftQuotient} 
\ips{\cdot}{\cdot}_{\Omega}^{} : \frac{\Omega^k(M)}{\OmegaZ^k(M)} \times \Omega^{m-k}_{{\rm c},\bbZ}(M)
\longrightarrow \bbT~, \qquad ([\omega], \alpha) \longmapsto (-1)^k\,\int_M\, \omega\wedge\alpha \mod \bbZ
\end{equation}
and by partial evaluation the two homomorphisms
\begin{subequations}\label{eqn:partevalOmega2}
\begin{flalign}
\label{eqn:partialevalOmegamodOmegaZ}\frac{\Omega^k(M)}{\OmegaZ^k(M)}  \longrightarrow \Omega^{m-k}_{{\rm c},\bbZ}(M)^\star~,& \qquad [\omega] \longmapsto \ips{[\omega]}{\cdot}_{\Omega}^{}~,\\[4pt]
\label{eqn:partialevalOmegacZ} \Omega^{m-k}_{{\rm c},\bbZ}(M) \longrightarrow \Big(\, \frac{\Omega^k(M)}{\OmegaZ^k(M)}\, \Big)^\star
~,& \qquad \alpha \longmapsto \ips{\cdot}{\alpha}_{\Omega}^{}~.
\end{flalign}
\end{subequations}
For manifolds $M$ of finite-type, Lemma \ref{lemCFormsZ} implies that
\begin{equation}\label{eqCFormsZ}
\Omega^{m-k}_{{\rm c},\bbZ}(M) 
= \Big\{\alpha \in \Omega^{m-k}_{\rm c}(M)\, :\; \int_M\, \Omega^{k}_\bbZ(M)\wedge \alpha \subseteq \bbZ\Big\}~. 
\end{equation}
We are now ready to demonstrate
\begin{lem}\label{lem:auxisosOmega}
The homomorphisms \eqref{eqn:partialevalOmegaZ} and \eqref{eqn:partialevalOmegamodOmegaZ} induce isomorphisms
\begin{equation}
\OmegaZ^k(M) \simeq \Big(\, \frac{\Omega^{m-k}_{\rm c}(M)}{\Omega^{m-k}_{{\rm c},\bbZ}(M)}\, \Big)^\star_\infty~,\qquad
\frac{\Omega^k(M)}{\OmegaZ^k(M)} \simeq \Omega^{m-k}_{{\rm c},\bbZ}(M)^\star_\infty~.
\end{equation}
For $M$ of finite-type, the homomorphisms \eqref{eqn:partialevalOmegacmodOmegacZ}  and
\eqref{eqn:partialevalOmegacZ} induce isomorphisms
\begin{equation}
\frac{\Omega^{m-k}_{\rm c}(M)}{\Omega^{m-k}_{{\rm c},\bbZ}(M)} \simeq \OmegaZ^k(M)^\star_\infty ~,\qquad 
\Omega^{m-k}_{{\rm c},\bbZ}(M)\simeq \Big(\, \frac{\Omega^k(M)}{\OmegaZ^k(M)}\, \Big)^\star_\infty~.
\end{equation}
\end{lem}
\begin{proof}
The first part follows from \eqref{eqFormsZ} by a straightforward calculation and the second part
similarly by using also \eqref{eqCFormsZ}.
\end{proof}

Following \cite{HLZ}, we finally define the smooth Pontryagin duals of 
(compactly supported) differential cohomology.
\begin{defi}\label{defSmoothPontrDualDiffChar}
\begin{itemize}
\item[(i)] The smooth Pontryagin dual $\dH^k_{\rm c}(M;\bbZ)^\star_\infty \subseteq \dH^k_{\rm c}(M;\bbZ)^\star$
of $\dH^k_{\rm c}(M;\bbZ)$ is the preimage
\begin{equation}
\dH^k_{\rm c}(M;\bbZ)^\star_\infty
	:= (\iota^\star)^{-1}\Big(\, \frac{\Omega_{\rm c}^{k-1}(M)}{\Omega^{k-1}_{{\rm c},\bbZ}(M)}\, \Big)^\star_\infty~,
\end{equation}
where $\iota^\star := \Hom(\iota,\bbT) :  \dH^k_{\rm c}(M;\bbZ)^\star \to (\Omega_{\rm c}^{k-1}(M) / \Omega^{k-1}_{{\rm c},\bbZ}(M))^\star$
is the Pontryagin dual of the topological trivialization $\iota: \Omega_{\rm c}^{k-1}(M) / \Omega^{k-1}_{{\rm c},\bbZ}(M) \to \dH^k_{\rm c}(M;\bbZ)$.

\item[(ii)] The smooth Pontryagin dual $\dH^k(M;\bbZ)^\star_\infty\subseteq \dH^k(M;\bbZ)^\star$
of $\dH^k(M;\bbZ)$ is the preimage
\begin{equation}
\dH^k(M;\bbZ)^\star_\infty
				:= (\iota^\star)^{-1}\Big(\, \frac{\Omega^{k-1}(M)}{\OmegaZ^{k-1}(M)}\, \Big)^\star_\infty~,
\end{equation}
where $\iota^\star := \Hom(\iota,\bbT) :  \dH^k(M;\bbZ)^\star \to (\Omega^{k-1}(M) / \Omega^{k-1}_{\bbZ}(M))^\star$
is the Pontryagin dual of the topological trivialization $\iota: \Omega^{k-1}(M) / \Omega^{k-1}_{\bbZ}(M) \to \dH^k(M;\bbZ)$.
\end{itemize}
\end{defi}

\paragraph*{Functoriality:}
We show that the smooth Pontryagin duals of (compactly supported) differential characters
define functors
\begin{equation}
\dH^k(-;\bbZ)^\star_\infty: \oMan_{m,\emb} \longrightarrow \Ab~,\qquad
\dH^k_{\rm c}(-;\bbZ)^\star_\infty: \oMan_{m,\emb}^\op \longrightarrow \Ab~.
\end{equation}
These functors are subfunctors
of $\dH^k(-;\bbZ)^\star := \Hom(\dH^k(-;\bbZ),\bbT): \oMan_{m,\emb} \to \Ab$ 
and $\dH^k_{\rm c}(-;\bbZ)^\star := \Hom(\dH^k_{\rm c}(-;\bbZ),\bbT): \oMan_{m,\emb}^\op \to \Ab$.
For any morphism $f: M \to M^\prime$ in $\oMan_{m,\emb}$, 
the corresponding push-forward $f_\ast := \Hom(f^\ast,\bbT): \dH^k(M;\bbZ)^\star \to \dH^k(M^\prime;\bbZ)^\star$ 
maps smooth group characters $\psi \in \dH^k(M;\bbZ)^\star_\infty$ 
to smooth group characters $f_\ast \psi \in \dH^k(M^\prime;\bbZ)^\star_\infty$
because $(\Omega^{k-1}(-)/\OmegaZ^{k-1}(-))^\star_\infty : \oMan_{m,\emb} \to \Ab$
is a functor (this follows from \eqref{eqSmoothPontrDualTT}) and $\iota^\star$ is by construction
a natural transformation.
Similarly, the pull-back $f^\ast := \Hom(f_\ast,\bbT): \dH^k_{\rm c}(M^\prime;\bbZ)^\star \to \dH^k_{\rm c}(M;\bbZ)^\star$ 
maps smooth group characters $\varphi^\prime \in \dH^k_{\rm c}(M^\prime;\bbZ)^\star_\infty$ to smooth group
characters $f^\ast\varphi^\prime \in \dH^k_{\rm c}(M;\bbZ)^\star_\infty$ because
$(\Omega_{\rm c}^{k-1}(-) / \Omega^{k-1}_{{\rm c},\bbZ}(-))^\star_\infty : 
\oMan_{m,\emb}^\op \to \Ab$ is a functor (this follows again from \eqref{eqSmoothPontrDualTT}) 
and $\iota^\star$ is by construction a natural transformation.

\subsection{Duality theorem}
We will characterize $\dH^k_{\rm c}(M;\bbZ)^\star_\infty$ 
in terms of differential characters and $\dH^k(M;\bbZ)^\star_\infty$ 
in terms of differential characters with compact support. This will establish
a version of smooth Pontryagin duality for
(compactly supported) differential characters. 
The main results of this section, Theorem \ref{thmCDiffCharIso} and Corollary \ref{corPontryaginDuality}, 
provide an independent proof of \cite[Theorem\ 8.7]{HLZ} within our model 
for compactly supported differential characters. This is a valuable result because our model 
is very close to the original Cheeger-Simons theory and does not rely on techniques from de Rham-Federer theory.
For results concerning Lefschetz-Pontryagin duality for de Rham-Federer characters refer to \cite{HL2}. 
\sk

The module structure on compactly supported differential characters
given in Section \ref{secCDiffChar} defines a bihomomorphism
$\cdot : \dH^k(M;\bbZ) \times\dH^{m-k+1}_{\rm c}(M;\bbZ) \to \dH^{m+1}_{\rm c}(M;\bbZ)$.
Using the diagram in Theorem \ref{thmCDiffChar},
we observe that $\dH^{m+1}_{\rm c}(M;\bbZ)$ is canonically isomorphic
to $\H^m_{\rm c}(M;\bbR) / \H^m_{{\rm c},\free}(M;\bbZ)$
and, since we assume $M$ to be connected,\footnote{
The following constructions can be extended to disconnected manifolds by treating each connected component separately.
For the sake of simplicity we do not consider this more general scenario.
} we obtain an isomorphism $\H^m_{\rm c}(M;\bbR) / \H^m_{{\rm c},\free}(M;\bbZ) \simeq \bbT$.
Hence this bihomomorphism maps to the circle group and defines a $\bbT$-valued pairing
\begin{equation}\label{eqPairingDiffChar}
\ips{\cdot}{\cdot}_{\rm c}:\dH^k(M;\bbZ) \times \dH^{m-k+1}_{\rm c}(M;\bbZ) \longrightarrow \bbT~, \qquad 
(h,h'\,) \longmapsto (h \cdot h'\,) \mu~.
\end{equation}
Here we have also given an explicit expression for the isomorphisms above,
which should be interpreted as follows: For $h'\, \in \dH^{m-k+1}_{\rm c}(M;\bbZ)$
we choose a representative in the colimit (denoted with abuse of notation by the same symbol) 
$h'\,\in \dH^{m-k+1}(M,M \setminus K;\bbZ)$ for some compact $K\subseteq M$.
Then $h \cdot h'\, \in \dH^{m+1}(M,M \setminus K;\bbZ)$ is a differential character on relative
cycles and $(h \cdot h'\,)\mu$ denotes its evaluation on (some representative of)
the unique relative homology class $\mu \in \H_m(M,M \setminus K)$
which restricts to the orientation of $M$ for each point of $K$, cf.\ \cite[Lemma 3.27]{Hatcher}. 
The pairing \eqref{eqPairingDiffChar} is natural in the sense that
for all morphisms $f:M \to M^\prime$ in $\oMan_{m,\emb}$
the diagram
\begin{equation}\label{eqPairingDiffCharNat}
\xymatrix@C=45pt{
\dH^k(M^\prime;\bbZ) \times \dH^{m-k+1}_{\rm c}(M;\bbZ) \ar[r]^-{f^\ast \times \id} \ar[d]_-{\id \times f_\ast}
	&	\dH^k(M;\bbZ) \times \dH^{m-k+1}_{\rm c}(M;\bbZ) \ar[d]^-{\ips{\cdot}{\cdot}_{\rm c}}						\\
\dH^k(M^\prime;\bbZ) \times \dH^{m-k+1}_{\rm c}(M^\prime;\bbZ) \ar[r]_-{\ips{\cdot}{\cdot}_{\rm c}}
	&	\bbT
}
\end{equation}
commutes; this is a consequence of naturality of the module structure 
\eqref{eqCDiffCharModuleNat} and uniqueness of the relative homology 
class representing the orientation,  see \cite[Lemma 3.27]{Hatcher}. 
By partial evaluation, the pairing \eqref{eqPairingDiffChar} defines the two homomorphisms
\begin{subequations}\label{eqn:partevaldH}
\begin{flalign}
\dH^k(M;\bbZ) \longrightarrow \dH^{m-k+1}_{\rm c}(M;\bbZ)^\star ~,&~~h \longmapsto \ips{h}{\cdot}_{\rm c}~,\\[4pt]
 \dH^{m-k+1}_{\rm c}(M;\bbZ) \longrightarrow \dH^k(M;\bbZ)^\star ~,&~~h'\, \longmapsto \ips{\cdot}{h'\,}_{\rm c}~.
\end{flalign}
\end{subequations}

Let us introduce another $\bbT$-valued pairing
\begin{equation}
\label{eqPairingCoho} \ips{\cdot}{\cdot}_{\H}^{} \,:\, \H^k(M;\bbT) \times \H^{m-k}_{\rm c}(M;\bbZ) \longrightarrow \bbT~, \qquad 
(u,c)					\longmapsto (u \smile c)\mu ~,
\end{equation}
between cohomology and compactly supported cohomology, which
is defined as in \eqref{eqPairingDiffChar}
by choosing a representative of the colimit $c \in \H^{m-k}(M,M \setminus K;\bbZ)$, for some $K\subseteq M$ compact, 
and evaluating $u \smile c\in \H^{m}(M,M \setminus K;\bbT)$ 
on the unique element $\mu \in \H_m(M,M \setminus K)$  
which restricts to the orientation of $M$ at each point of $K$. 
Partial evaluation provides the homomorphisms
\begin{subequations}\label{eqn:partevalH}
\begin{flalign}
\label{eqn:parteval3}\H^k(M;\bbT) \longrightarrow \H^{m-k}_{\rm c}(M;\bbZ)^\star ~,&~~ u\longmapsto \ips{u}{\cdot}_{\H}^{}~,\\[4pt]
\label{eqn:parteval4}\H^{m-k}_{\rm c}(M;\bbZ) \longrightarrow \H^k(M;\bbT)^\star ~,&~~c\longmapsto \ips{\cdot}{c}_{\H}^{}~.
\end{flalign}
\end{subequations}
\begin{lem}\label{lem:auxisosH}
Let $M$ be an object of $\oMan_{m,\emb}$. Then the homomorphism \eqref{eqn:parteval3} is an isomorphism.
For $M$ of finite-type, the homomorphism \eqref{eqn:parteval4} is also an isomorphism.
\end{lem}
\begin{proof}
The first statement follows from Poincar\'e duality $\H_{\rm c}^{m-k}(M;\bbZ)\simeq \H_k(M)$, 
see e.g.~\cite[Theorem 3.35]{Hatcher}, and the fact that 
\begin{equation}\label{eqn:tmp}
\H^k(M;\bbT) \simeq \Hom(\H_k(M),\bbT)=\H_k(M)^\star~,
\end{equation}
which is a consequence of the universal coefficient theorem for cohomology and divisibility of $\bbT$.
The second statement follows by taking the Pontryagin dual of \eqref{eqn:tmp} and recalling that Pontryagin
duality is reflexive for finitely generated Abelian groups, in particular
on all (co)homology groups of manifolds $M$ of finite-type.
\end{proof}

\begin{theo}\label{thmCDiffCharIso}
Let $M$ be an object of $\oMan_{m,\emb}$. Then the diagram 
\begin{equation}
\xymatrix@C=35pt{
0 \ar[r]	&	\H^{m-k}(M;\bbT) \ar[r]^-\kappa \ar[d]_\simeq
				&	\dH^{m-k+1}(M;\bbZ) \ar[r]^-\cu \ar[d]_\simeq
						&	 \OmegaZ^{m-k+1}(M) \ar[r] \ar[d]^\simeq							&	0	\\
0 \ar[r]	&	\H^k_{\rm c}(M;\bbZ)^\star \ar[r]_-{\ch^\star}
				&	\dH^k_{\rm c}(M;\bbZ)^\star_\infty \ar[r]_-{\iota^\star}
						&	\Big(\, \frac{\Omega^{k-1}_{\rm c}(M)}
								{\Omega^{k-1}_{{\rm c},\bbZ}(M)}\, \Big)^\star_\infty \ar[r]					&	0 
}
\end{equation}
commutes, its rows are short exact sequences and the vertical arrows are natural isomorphisms.
For $M$ of finite-type, the diagram
\begin{equation}
\xymatrix@C=35pt{
0 \ar[r]	&	\frac{\Omega^{m-k}_{\rm c}(M)}{\Omega^{m-k}_{{\rm c},\bbZ}(M)} \ar[r]^-\iota \ar[d]_\simeq
				&	\dH^{m-k+1}_{\rm c}(M;\bbZ) \ar[r]^-\ch \ar[d]_\simeq
						&	 \H^{m-k+1}_{\rm c}(M;\bbZ) \ar[r] \ar[d]^\simeq							&	0	\\
0 \ar[r]	&	\OmegaZ^k(M)^\star_\infty \ar[r]_-{\cu^\star}
				&	\dH^k(M;\bbZ)^\star_\infty \ar[r]_-{\kappa^\star}
						&	\H^{k-1}(M;\bbT)^\star \ar[r]										&	0
}
\end{equation}
commutes, its rows are short exact sequences and the vertical arrows are natural isomorphisms.
In both diagrams the vertical arrows are the partial evaluations given in \eqref{eqn:partevalOmega1},
\eqref{eqn:partevaldH} and \eqref{eqn:partevalH}.
\end{theo}
\begin{proof}
The top row in the first diagram is the middle row of \eqref{eqDiffCharDia}
and the top row in the second diagram is the middle column in 
\eqref{eqCDiffCharDia}. Hence they are short exact sequences.
The bottom row in the second diagram is a short exact sequence by
\cite[Theorem 4.3]{Becker:2014tla}. We show that the bottom row
in the first diagram is a short exact sequence. Now $\ch^\star : \H^k_{\rm c}(M;\bbZ)^\star \to \dH^k_{\rm c}(M;\bbZ)^\star$ 
maps (injectively) to smooth group characters as a consequence of $\ch \circ \iota = 0$
and $\iota^\star: \dH^k_{\rm c}(M;\bbZ)^\star \to (\Omega^{k-1}_{\rm c}(M) / \Omega^{k-1}_{{\rm c},\bbZ}(M))^\star$ 
maps (surjectively because of 
Definition \ref{defSmoothPontrDualDiffChar}) smooth group characters to smooth group characters.
Exactness at the middle object follows from the fact that
$\Hom(-,\bbT): \Ab^\op \to \Ab$ is an exact functor as $\bbT$ is divisible.
\sk

By Lemma \ref{lem:auxisosOmega} and Lemma \ref{lem:auxisosH}, 
the left and right vertical arrows in both diagrams are isomorphisms.
Commutativity of both diagrams can be shown by using the properties \eqref{eqCCompatibility}
of the module structure on compactly supported differential characters.
Using the five lemma, we conclude that the middle vertical 
arrow in each diagram is also an isomorphism.
\end{proof}

This theorem immediately implies 
\begin{cor}[Smooth Pontryagin duality for differential characters]\label{corPontryaginDuality}
The partial evaluations in \eqref{eqn:partevaldH} define a natural isomorphism
between the functors
\begin{flalign}
\dH^{m-k+1}(-;\bbZ) : \oMan_{m,\emb}^\op \longrightarrow \Ab~,& \qquad 
\dH^k_{\rm c}(-;\bbZ)^\star_\infty : \oMan_{m,\emb}^\op \longrightarrow \Ab~,\\
\intertext{and a natural isomorphism between the functors}
\dH^{m-k+1}_{\rm c}(-;\bbZ) : \oMan_{m,\emb,\mathrm{ft}} \longrightarrow \Ab~,& \qquad
\dH^k(-;\bbZ)^\star_\infty : \oMan_{m,\emb,\mathrm{ft}} \longrightarrow \Ab~,
\end{flalign}
where $\oMan_{m,\emb,\mathrm{ft}}$ is the full subcategory of $\oMan_{m,\emb}$
whose objects are manifolds of finite-type.
\end{cor}
\begin{cor}[Smooth Pontryagin duality: Pairing version]\label{corNonDegC}
Let $M$ be an oriented and connected $m$-dimensional manifold of finite-type. Then the pairing 
\begin{equation}
\ips{\cdot}{\cdot}_{\rm c}:\dH^k(M;\bbZ) \times \dH^{m-k+1}_{\rm c}(M;\bbZ) \longrightarrow \bbT
\end{equation}
introduced in \eqref{eqPairingDiffChar} is weakly non-degenerate. 
\end{cor}

\begin{rem}\label{remPontryaginIso}
By introducing suitable pairings, one can easily extend Theorem \ref{thmCDiffCharIso} 
to the full diagrams for (compactly supported) differential characters 
given in \eqref{eqDiffCharDia} and \eqref{eqCDiffCharDia}. 
For example, the second diagram in Theorem \ref{thmCDiffCharIso} (for $M$ of finite-type)
extends to the three-dimensional commutative diagram
\begin{equation}
{\small
\xymatrix@C=0pt@R=1pt{
		&&	0 \ar@{.>}[dd]	&&	0 \ar@{.>}[dd]	&&	0 \ar@{.>}[dd]													\\
&			&&	0 \ar[dd]	&&	0 \ar[dd]	&&	0 \ar[dd]												\\
0 \ar@{.>}[rr]	&&	\frac{\H^{m-k}_{\rm c}(M;\bbR)}{\H_{{\rm c},\free}^{m-k}(M;\bbZ)} \ar@{.>}'[r][rr] \ar@{.>}'[d][dd] \ar[rd]^-\simeq
						&&	\frac{\Omega^{m-k}_{\rm c}(M)}{\Omega^{m-k}_{{\rm c},\bbZ}(M)} 
								\ar@{.>}'[r][rr] \ar@{.>}'[d][dd]^-\iota \ar[rd]^-\simeq
										&&	\dd \Omega^{m-k}_{\rm c}(M) \ar@{.>}'[r][rr] \ar@{.>}'[d][dd] \ar[rd]^-\simeq	
																									&&	0\\
&	0 \ar[rr]	&&	\Hf^k(M;\bbZ)^\star \ar[rr] \ar[dd]
							&&	\Omega^k_\bbZ(M)^\star_\infty \ar[rr] \ar[dd]^(.3){\cu^\star}
											&&	(\dd \Omega^{k-1}(M))^\star_\infty \ar[rr] \ar[dd]	&&	0\\
0 \ar@{.>}[rr]	&&	\H^{m-k}_{\rm c}(M;\bbT) \ar@{.>}'[r][rr]_-\kappa \ar@{.>}'[d][dd] \ar[rd]^-\simeq
						&&	\dH^{m-k+1}_{\rm c}(M;\bbZ) \ar@{.>}'[r][rr]_-\cu \ar@{.>}'[d][dd]^-\ch \ar[rd]^-\simeq
										&&	\Omega^{m-k+1}_{{\rm c},\bbZ}(M)
												\ar@{.>}'[r][rr] \ar@{.>}'[d][dd] \ar[rd]^-\simeq					&&	0\\
&	0 \ar[rr]	&&	\H^k(M;\bbZ)^\star \ar[rr]_(.4){\ch^\star} \ar[dd]	
							&&	\dH^k(M;\bbZ)^\star_\infty
									\ar[rr]_(.4){\iota^\star} \ar[dd]^(.3){\kappa^\star}
											&&	\left(\frac{\Omega^{k-1}(M)}{\Omega^{k-1}_\bbZ(M)}
													\right)^\star_\infty \ar[rr] \ar[dd]				&&	0\\
0 \ar@{.>}[rr]	&&	\H^{m-k+1}_{{\rm c},\tor}(M;\bbZ) \ar@{.>}'[r][rr] \ar@{.>}'[d][dd] \ar[rd]^-\simeq
						&&	\H^{m-k+1}_{\rm c}(M;\bbZ) \ar@{.>}'[r][rr] \ar@{.>}'[d][dd] \ar[rd]^-\simeq
										&&	\H^{m-k+1}_{{\rm c},\free}(M;\bbZ)
												\ar@{.>}'[r][rr] \ar@{.>}'[d][dd] \ar[rd]^-\simeq					&&	0\\
&	0 \ar[rr]	&&	\Ht^k(M;\bbZ)^\star \ar[rr] \ar[dd]
							&&	\H^{k-1}(M;\bbT)^\star \ar[rr] \ar[dd]
											&&	\left(\frac{\H^{k-1}(M;\bbR)}{\Hf^{k-1}(M;\bbZ)}
													\right)^\star \ar[rr] \ar[dd]						&&	0\\
		&&	0		&&	0		&&	0																	\\
&			&&	0			&&	0			&&	0
}}
\end{equation}
where all diagonal arrows are isomorphisms, the foreground face is the smooth Pontryagin dual of 
\eqref{eqDiffCharDia} and the background face is given by \eqref{eqCDiffCharDia}.
\end{rem}

\begin{rem}\label{remNonSub}
We compare our results to \cite[Theorem 8.4]{HLZ}. 
As mentioned in Remark~\ref{remHLZ}, de~Rham-Federer characters 
with compact support are introduced in \cite{HLZ} by means of 
de~Rham-Federer currents. The group of de~Rham-Federer characters 
with compact support is isomorphic to the smooth Pontryagin dual 
of the group of de~Rham-Federer characters, cf.~\cite[Theorem 8.7]{HLZ}.
According to \cite[Theorem 4.1]{HLZ}, the latter is isomorphic to the group 
$\dH^\sharp(M;\bbZ)$ of Cheeger-Simons differential characters, therefore
we deduce from Theorem~\ref{thmCDiffCharIso} that the group of 
de~Rham-Federer characters with compact support is isomorphic to the 
group $\dH^\sharp_{\rm c}(M;\bbZ)$ of compactly supported differential characters 
introduced in Definition~\ref{defCDiffChar}.
Together with Remark \ref{remNonInj}, this contradicts \cite[Theorem 8.4]{HLZ}, 
which states that the group of de~Rham-Federer characters with compact 
support is isomorphic to a \emph{subgroup} of the group of 
Cheeger-Simons differential characters. 
As already observed in Remark \ref{remHLZ}, the failure of the homomorphism in \cite[Theorem 8.4]{HLZ} 
to be an isomorphism does not affect the rest of \cite{HLZ}. 
\end{rem}

\subsection{Pairing between differential characters with compact support}\label{subPairing}
We conclude by defining a $\bbT$-valued pairing on
differential characters with compact support and describe its properties.
Let $M$ be any object of $\oMan_{m,\emb}$. 
Using the homomorphism $I :\dH^k_{\rm c}(M;\bbZ) \to \dH^k(M;\bbZ)$ defined in
\eqref{eqCDiffCharToDiffChar}, we introduce the pairing
\begin{equation}\label{eqPairingCDiffChar}
\ips{I \cdot}{\cdot}_{\rm c}: \dH^k_{\rm c}(M;\bbZ) \times \dH^{m-k+1}_{\rm c}(M;\bbZ)
															\longrightarrow
                                                                                                                        \bbT~,
                                                                                                                        \qquad 
(h,h'\,)														\longmapsto		\ips{I h}{h'\,}_{\rm c}~
\end{equation}
on compactly supported differential characters.
\begin{propo}
The pairing \eqref{eqPairingCDiffChar} is graded symmetric, i.e.\
\begin{flalign}
\ips{I h}{h'\,}_{\rm c} = (-1)^{k\,(m-k+1)}\, \ips{I h'\,}{h}_{\rm c}~,
\end{flalign}
for all $h\in \dH^k_{\rm c}(M;\bbZ)$ and $h'\,\in \dH^{m-k+1}_{\rm c}(M;\bbZ)$.
\end{propo}
\begin{proof}
This result is a consequence of the graded commutative (possibly non-unital) ring structure 
on relative differential cohomology, see Remark \ref{remRelDiffCharMultiplication}:
Given $h \in \dH^k_{\rm c}(M;\bbZ)$ and $h'\, \in \dH^{m-k+1}_{\rm c}(M;\bbZ)$
there exists $K\subseteq M$ compact such that $h \in \dH^k(M,M \setminus K;\bbZ)$ 
and $h'\, \in \dH^{m-k+1}(M,M \setminus K;\bbZ)$ are representatives in the
corresponding colimits. 
Using \eqref{eqRelModuleAction} one shows the identities 
$I h \cdot h'\, = h \cdot h'\,$ and $I h'\, \cdot h = h'\, \cdot h$ 
of elements in $\dH^{m+1}(M,M \setminus K;\bbZ)$, 
where on the left hand sides $\cdot$ denotes the $\dH^\sharp(M;\bbZ)$-module structure 
on $\dH^\sharp(M,M\setminus K;\bbZ)$ (see \eqref{eqRelDiffCharModule})
and on the right hand sides $\cdot$ denotes the ring structure
on $\dH^\sharp(M,M\setminus K;\bbZ)$ 
(see \eqref{eqRelDiffCharProd} for $S = S^\prime = M \setminus K$). 
As a consequence of graded commutativity of the ring structure 
on $\dH^\sharp(M,M\setminus K;\bbZ)$,
we obtain $I h \cdot h'\, = h \cdot h'\, = (-1)^{k\, (m-k+1)}\, h'\,
\cdot h = (-1)^{k\, (m-k+1)}\, I h'\, \cdot h$
and the result follows by recalling \eqref{eqPairingDiffChar}.
\end{proof}

\begin{rem}\label{remRadical}
Unlike $\ips{\cdot}{\cdot}_{\rm c}$, the pairing $\ips{I \cdot}{\cdot}_{\rm c}$ given in 
\eqref{eqPairingCDiffChar} might be degenerate, even if $M$ is of finite-type.
This is because the homomorphism $I: \dH^k_{\rm c}(M;\bbZ) \to \dH^k(M;\bbZ)$ in general 
fails to be injective, cf.\ Remark \ref{remNonInj}. 
For $M$ of finite-type, Corollary \ref{corNonDegC} implies that
the degeneracy of the pairing \eqref{eqPairingCDiffChar} coincides precisely with $\ker I$.
\end{rem}

We finish by proving that the pairing \eqref{eqPairingCDiffChar} is natural.
\begin{propo}\label{prpPairingCDiffCHarNat}
For any morphism $f: M \to M^\prime$ in $\oMan_{m,\emb}$ the diagram
\begin{equation}
\xymatrix{
\dH^k_{\rm c}(M;\bbZ) \times \dH^{m-k+1}_{\rm c}(M;\bbZ)  \ar[rd]_-{\ips{I \cdot}{\cdot}_{\rm c}}		\ar[rr]^-{f_\ast \times f_\ast} &&
\dH^k_{\rm c}(M^\prime;\bbZ) \times \dH^{m-k+1}_{\rm c}(M^\prime;\bbZ) \ar[ld]^-{\ips{I \cdot}{\cdot}_{\rm c}}		 \\
&\bbT&
}
\end{equation}
commutes.
\end{propo}
\begin{proof}
This is a direct consequence of naturality of 
$\ips{\cdot}{\cdot}_{\rm c}: \dH^k(M;\bbZ) \times \dH^{m-k+1}_{\rm c}(M;\bbZ) \to \bbT$,
cf.\ \eqref{eqPairingDiffCharNat},
and naturality of $I: \dH^k_{\rm c}(M;\bbZ) \to \dH^k(M;\bbZ)$,
cf.\  \eqref{eqCDiffCharToDiffCharNat}: The short calculation
\begin{equation}
\ips{I \, f_\ast \cdot}{f_\ast \cdot}_{\rm c} = \ips{f^\ast \, I \, f_\ast \cdot}{\cdot}_{\rm c} = \ips{I\cdot}{\cdot}_{\rm c}~
\end{equation}
proves the claim.
\end{proof}


\section*{Acknowledgments}
It is a pleasure to thank Christian B\"ar for very helpful discussions 
and Ulrich Bunke for his valuable comments on the first version of this paper.
We are grateful to the anonymous referees for their useful suggestions, 
that contributed to improve the quality of the paper. 
This work was supported in part by the Action MP1405 QSPACE from the 
European Cooperation in Science and Technology (COST).  
The work of C.B.\ is partially supported by the Collaborative Research Center (SFB) 
``Raum Zeit Materie'', funded by the Deutsche Forschungsgemeinschaft (DFG, Germany).
The work of M.B.\ is supported partly by a Research Fellowship of the Della Riccia Foundation (Italy) 
and partly by a Postdoctoral Fellowship of the Alexander von Humboldt Foundation (Germany). 
The work of A.S.\ is supported by a Research Fellowship of the Deutsche Forschungsgemeinschaft (DFG, Germany). 
The work of R.J.S.\ is partially supported by the Consolidated Grant ST/L000334/1 
from the UK Science and Technology Facilities Council.


\appendix

\section{Technical details for Section \ref{secPontrDuality}}\label{app:axuiliary}
We prove two lemmas which are used in Section \ref{secPontrDuality} 
to obtain the explicit characterizations \eqref{eqCFormsZ} and \eqref{eqFormsZ}
of (compactly supported) forms with integral periods.
For the first lemma we consider manifolds of finite-type only, 
while in the second lemma there is no such restriction. 

\begin{lem}\label{lemCFormsZ}
Let $M$ be an oriented $m$-dimensional manifold of finite-type.
Let $\alpha \in \Omega^k_{\rm c}(M)$ and recall the definition \eqref{eqRelFormsZ}. 
Then the following conditions are equivalent:
\begin{enumerate}
\item There exists $K \subseteq M$ compact such that $\alpha \in \OmegaZ^k(M,M \setminus K)$; 
\item $\int_M\, \omega \wedge \alpha \in \bbZ$, for each $\omega \in \OmegaZ^{m-k}(M)$. 
\end{enumerate}
\end{lem}
\begin{proof}
$1 \Rightarrow 2$: Let $\alpha \in \Omega^k_\bbZ(M,M \setminus K)$ for 
a compact subset $K \subseteq M$. 
By the definition of $\Omega^k_\bbZ(M,M \setminus K)$ in \eqref{eqRelFormsZ}, 
the cochain $\int_\cdot\, \alpha \in C^k(M,M \setminus K;\bbR)$ induces a relative cohomology class 
$\int_\cdot \alpha \in \Hom(\H_k(M,M \setminus K),\bbZ)) \simeq \Hf^k(M,M \setminus K;\bbZ)$. 
Passing to the colimit, we interpret $\int_\cdot\, \alpha$ as an element of $\H^k_{{\rm c},\free}(M;\bbZ)$. 
Given $\omega \in \Omega^{m-k}_\bbZ(M)$, the cochain $\int_\cdot\, \omega \in C^{m-k}(M;\bbR)$ 
induces a cohomology class $\int_\cdot \omega \in \Hom(\H_{m-k}(M),\bbZ) \simeq \Hf^{m-k}(M;\bbZ)$. 
Similarly to \cite[Proposition 3.38]{Hatcher}, 
the cup product for singular cohomology provides an integer-valued pairing 
\begin{equation}\label{eqFreeCohoPairing}
\H_\free^{m-k}(M;\bbZ)  \times \H^k_{{\rm c},\free}(M;\bbZ) \longrightarrow \bbZ~,
\end{equation}
which gives $(\, \int_\cdot\, \omega \smile \int_\cdot\, \alpha\, ) \mu \in \bbZ$ 
upon evaluation on $(\, \int_\cdot\, \omega,\int_\cdot\, \alpha\, )$. 
Here $\mu \in \H_m(M,M \setminus K)$ denotes the unique relative homology class 
which restricts to the orientation of $M$ at each point of $K$. 
Since $\wedge$ and $\smile$ are naturally cochain homotopic on differential forms,
we have $\int_M\, \omega \wedge \alpha = (\, \int_\cdot \, \omega
\smile \int_\cdot\, \alpha\, )\mu \in \bbZ$. 
\sk

$2 \Rightarrow 1$: Let $\alpha \in \Omega^k_{\rm c}(M)$ with 
$\int_M\, \omega \wedge \alpha \in \bbZ$, for all $\omega \in \Omega^{m-k}_\bbZ(M)$,
and denote the support of $\alpha$ by $K^\prime:= \supp\, \alpha \subseteq M$. 
We have to show that $\int_\cdot\, \alpha \in C^k(M;\bbR)$ 
induces a $\bbZ$-valued homomorphism on $Z_k(M,M \setminus K)$ for some $K \subseteq M$ compact. 
By our assumptions, $\int_M\, \cdot \wedge \alpha$ defines a $\bbZ$-valued homomorphism 
on $\Omega^{m-k}_\bbZ(M) / \dd \Omega^{m-k-1}(M) \simeq \Hf^{m-k}(M;\bbZ)$. 
Taking into account the pairing \eqref{eqFreeCohoPairing} 
and recalling that $\wedge$ and $\smile$ are naturally cochain homotopic on differential forms, 
we obtain $(\, \int_\cdot \, \omega \smile \int_\cdot\, \alpha\, )\mu = \int_M\, \omega \wedge \alpha$, 
for all $\omega \in \Omega^{m-k}_\bbZ(M)$.
As a consequence, we have $( \cdot\smile \int_\cdot\,\alpha)\mu \in \Hom(\Hf^{m-k}(M;\bbZ),\bbZ)$. 
By \cite[Proposition 3.38]{Hatcher} extended to manifolds of finite-type, 
the pairing induced by the cup product is perfect. In particular, 
it provides an isomorphism $\H^k_{{\rm c},\free}(M;\bbZ) \simeq \Hom(\Hf^{m-k}(M;\bbZ),\bbZ)$, 
hence $\int_\cdot \, \alpha \in \H^k_{{\rm c},\free}(M;\bbZ)$. Recalling the definition 
of $\H^k_{{\rm c},\free}(M;\bbZ)$ in terms of a colimit, one finds $K \subseteq M$ compact such that 
$\int_\cdot\, \alpha \in \Hf^k(M,M \setminus K;\bbZ) \simeq \Hom(\H_k(M,M \setminus K),\bbZ)$. 
The implication then follows from the obvious inclusion 
$\Hom(\H_k(M,M \setminus K),\bbZ) \subseteq \Hom(Z_k(M,M \setminus K),\bbZ)$. 
\end{proof}

\begin{lem}\label{lemFormsZ}
Let $M$ be an oriented $m$-dimensional manifold (not necessarily of finite-type). 
Let $\omega \in \Omega^k(M)$. Then the following conditions are equivalent: 
\begin{enumerate}
\item $\omega$ has integral periods, i.e.\ $\omega \in \OmegaZ^k(M)$; 
\item $\int_M\, \omega \wedge \alpha \in \bbZ$, for each $\alpha \in \Omega^{m-k}_{{\rm c},\bbZ}(M)$. 
\end{enumerate}
\end{lem}
\begin{proof}
Let $\omega \in \Omega^k(M)$ satisfy 
$\int_M\, \omega \wedge \Omega^{m-k}_{{\rm c},\bbZ}(M) \subseteq \bbZ$. 
Since $\dd \Omega_{\rm c}^{m-k-1}(M)$ is a vector space over $\bbR$ and also a subgroup of $\Omega^{m-k}_{{\rm c},\bbZ}(M)$, 
we deduce that $\int_M\, \omega \wedge \dd \Omega_{\rm c}^{m-k-1}(M) = \{0\}$, 
hence $\dd \omega = 0$ via Stokes' theorem. 
This implies that $\int_\cdot\, \omega \in C^k(M;\bbR)$ descends to a homomorphism on $\H_k(M)$. 
Recalling Poincar\'e duality, see e.g.\ \cite[Theorem~3.35]{Hatcher}, 
there is a natural isomorphism $\H_k(M) \simeq \H^{m-k}_{\rm c}(M;\bbZ)$.
Moreover, for each $x \in \H_k(M)$ one has $\int_x\, \omega = (\xi \smile \int_\cdot\, \omega)\mu$, 
where $\xi \in \H^{m-k}_{\rm c}(M;\bbZ)$ is the image of $x$ under Poincar\'e duality. On the right hand side
of this equation we have chosen a representative $\xi \in \H^{m-k}(M,M \setminus K;\bbZ)$  in the colimit,
for $K\subseteq M$ compact, and $\mu \in \H_m(M,M \setminus K)$ 
is the unique element which agrees with the orientation of $M$ at each point of $K$. 
For $\xi \in \H^{m-k}_{{\rm c},\tor}(M;\bbZ)$ we have $(\xi \smile \int_\cdot\, \omega)\mu = 0$
by a similar argument as in \cite[Proposition~3.38]{Hatcher}. 
Therefore, showing that $\omega$ has integral periods is equivalent to 
checking that the cup product between any $\xi \in \H^{m-k}_{{\rm c},\free}(M;\bbZ)$ and 
$\int_\cdot\, \omega \in \H^k(M;\bbR)$ is $\bbZ$-valued. 
Taking into account exactness of the right column of the diagram displayed in Theorem \ref{thmCDiffChar}, 
and recalling that $\wedge$ and $\smile$ are naturally cochain homotopic on differential forms, 
we conclude that $\omega$ has integral periods 
if and only if $\int_M\, \omega\wedge  \alpha \in \bbZ$, for all $\alpha \in \Omega^{m-k}_{{\rm c},\bbZ}(M)$.
This shows that $1 \Leftrightarrow 2$. 
\end{proof}


\end{document}